\documentclass[12pt]{amsart}
\usepackage{pgf,tikz,pgfplots} 
\pgfplotsset{compat=1.14}
\usepackage{mathrsfs}
\usetikzlibrary{arrows}
\usetikzlibrary{patterns}
\usepackage{pgfplots}
\usetikzlibrary{intersections, pgfplots.fillbetween}
\usepackage[colorlinks=true,urlcolor=blue, citecolor=red,linkcolor=blue,linktocpage,pdfpagelabels, bookmarksnumbered,bookmarksopen]{hyperref}
\usepackage[hyperpageref]{backref}
\usepackage{cleveref}
\usepackage{xcolor}
\usepackage{amsthm} 
\usepackage{latexsym,amsmath,amssymb}
\usepackage{accents}
\usepackage{a4wide}
\usepackage{soul}
\usepackage{mathtools} 
\usepackage{xparse} 
  
\pgfdeclarelayer{ft}
\pgfdeclarelayer{bg}
\pgfsetlayers{bg,main,ft}

\title[Stability for $W^{s,n/s}$-harmonic maps]{$s$-stability for $W^{s,n/s}$-harmonic maps\\ in homotopy groups}

\author{Katarzyna Mazowiecka}
\address[Katarzyna Mazowiecka]{
Institute of Mathematics,%
University of Warsaw,
Banacha 2,
02-097 Warszawa, Poland}
\email{k.mazowiecka@mimuw.edu.pl}
\author{Armin Schikorra}
\address[Armin Schikorra]{Department of Mathematics,
University of Pittsburgh,
301 Thackeray Hall,
Pittsburgh, PA 15260, USA}
\email{armin@pitt.edu}
\definecolor{indigo}{rgb}{0.29, 0.0, 0.51}
\definecolor{p1}{gray}{0.4}
\definecolor{p2}{gray}{0.6}
\definecolor{p3}{gray}{0.98}
\definecolor{p4}{gray}{0.8}
\definecolor{p5}{gray}{0.9}

plus 6pt minus 12pt
\def\eps{\varepsilon}

\def\C{{\mathbb C}}

\newcommand{\dif}{\,\mathrm{d}}

\def\N{{\mathbb N}}

\def\S{{\mathbb S}}

\renewcommand{\div}{{\rm div}}

\newtheorem{theorem}{Theorem}
\newtheorem{lemma}[theorem]{Lemma}
\newtheorem{corollary}[theorem]{Corollary}
\newtheorem{proposition}[theorem]{Proposition}

\newtheorem{remark}[theorem]{Remark}

\def\supp{{\rm supp\,}}

\newcommand{\dx}{\dif x}
\newcommand{\dy}{\dif y}
\newcommand{\dz}{\dif z}

\newcommand{\R}{\mathbb{R}}

\newcommand{\Z}{\mathbb{Z}}

\newcommand{\brac}[1]{\left (#1 \right )}
\newcommand{\abs}[1]{\left |#1 \right |}

\newcommand{\barint}{
\rule[.036in]{.12in}{.009in}\kern-.16in \displaystyle\int }

\newcommand{\barcal}{\mbox{$ \rule[.036in]{.11in}{.007in}\kern-.128in\int $}}

\def\mvint_#1{\mathchoice
          {\mathop{\vrule width 6pt height 3 pt depth -2.5pt
                  \kern -8pt \intop}\nolimits_{\kern -3pt #1}}%
          {\mathop{\vrule width 5pt height 3 pt depth -2.6pt
                  \kern -6pt \intop}\nolimits_{#1}}%
          {\mathop{\vrule width 5pt height 3 pt depth -2.6pt
                  \kern -6pt \intop}\nolimits_{#1}}%
          {\mathop{\vrule width 5pt height 3 pt depth -2.6pt
                  \kern -6pt \intop}\nolimits_{#1}}}

\numberwithin{theorem}{section} \numberwithin{equation}{section}

\newcommand{\lap}{\Delta }
\newcommand{\aleq}{\precsim}

\newcommand{\aeq}{\approx}

\newcommand{\Ds}[1]{(-\lap)^{\frac{#1}{2}}}
\newcommand{\laps}[1]{(-\lap)^{\frac{#1}{2}}}

\newcommand{\lapms}[1]{I^{#1}}

\let\latexchi\chi
\makeatletter
\renewcommand\chi{\@ifnextchar_\sub@chi\latexchi}
\newcommand{\sub@chi}[2]{
  \@ifnextchar^{\subsup@chi{#2}}{\latexchi^{}_{#2}}%
}
\newcommand{\subsup@chi}[3]{
  \latexchi_{#1}^{#3}%
}
\makeatother

\makeatletter
\def\tikz@arc@opt[#1]{
  {%
    \tikzset{every arc/.try,#1}%
    \pgfkeysgetvalue{/tikz/start angle}\tikz@s
    \pgfkeysgetvalue{/tikz/end angle}\tikz@e
    \pgfkeysgetvalue{/tikz/delta angle}\tikz@d
    \ifx\tikz@s\pgfutil@empty%
      \pgfmathsetmacro\tikz@s{\tikz@e-\tikz@d}
    \else
      \ifx\tikz@e\pgfutil@empty%
        \pgfmathsetmacro\tikz@e{\tikz@s+\tikz@d}
      \fi%
    \fi
    \tikz@arc@moveto
    \xdef\pgf@marshal{\noexpand%
    \tikz@do@arc{\tikz@s}{\tikz@e}
      {\pgfkeysvalueof{/tikz/x radius}}
      {\pgfkeysvalueof{/tikz/y radius}}}%
  }%
  \pgf@marshal%
  \tikz@arcfinal%
}
\let\tikz@arc@moveto\relax
\def\tikz@arc@movetolineto#1{%
  \def\tikz@arc@moveto{\tikz@@@parse@polar{\tikz@arc@@movetolineto#1}(\tikz@s:\pgfkeysvalueof{/tikz/x radius} and \pgfkeysvalueof{/tikz/y radius})}}
\def\tikz@arc@@movetolineto#1#2{#1{\pgfpointadd{#2}{\tikz@last@position@saved}}}
\tikzset{%
  move to start/.code=\tikz@arc@movetolineto\pgfpathmoveto,%
  line to start/.code=\tikz@arc@movetolineto\pgfpathlineto}
\makeatother

\begin{document}

\begin{abstract}
We study $s$-dependence for minimizing $W^{s,n/s}$-harmonic maps $u\colon \mathbb{S}^n \to \mathbb{S}^\ell$ in homotopy classes. Sacks--Uhlenbeck theory shows that, for each $s$, minimizers exist in a generating subset of $\pi_{n}(\S^\ell)$. We show that this generating subset can be chosen locally constant in $s$. We also show that as $s$ varies the minimal $W^{s,n/s}$-energy in each homotopy class changes continuously. In particular, we provide progress to a question raised by Mironescu \cite{Mironescu2007} and Brezis--Mironescu \cite{BM-book}.
\end{abstract}

\sloppy

\maketitle
\tableofcontents
\sloppy

\section{Introduction}
We study minimizing $W^{s,p}$-harmonic maps between spheres in homotopy classes, which are defined as maps $u \in W^{s,p}(\S^n,\S^\ell)$ with least energy
\begin{equation}\label{def:energy}
 \mathcal{E}_{s,p}(u) \coloneqq \int_{\S^n} \int_{\S^n} \frac{|u(x)-u(y)|^{p}}{|x-y|^{n+sp}}\dx\dy = [u]^p_{W^{s,p}(\S^n,\S^\ell)}
 \end{equation}
 among maps of the same homotopy.
Here  $s\in(0,1)$, $p>1$. In \eqref{def:energy} we take the $\R^{n+1}$-Euclidean distance and $\R^{\ell+1}$-Euclidean distance in the numerator and in the denominator, respectively.

A natural question arises: given $\alpha \in \pi_{n}(\S^\ell)$, is the infimum of $\mathcal{E}_{s,p}$ attained in $\alpha$? In other words, can we find a map $u\in W^{s,p}(\S^n,\S^\ell)$, $u \in \alpha$ such that
\[
\mathcal{E}_{s,p}(u) \leq   \mathcal{E}_{s,p}(v)\quad \forall v\in W^{s,p}(\S^n,\S^\ell), \quad v \in \alpha.
\]
If $p > \frac{n}{s}$ the answer is yes, by standard methods of calculus of variation, due to the compact embedding of $W^{s,p}(\S^n,\S^\ell)$ into $C^0(\S^n,\S^\ell)$. If $p < \frac{n}{s}$ it was shown by Brezis--Nirenberg in \cite{BrezisNirenbergI} that no homotopy theory can be defined for maps in $W^{s,p}(\S^n,\S^\ell)$ and the infimum energy of $\mathcal{E}_{s,p}$ in any homotopy class is zero.
%
Thus, throughout this work we will focus on the critical, conformally invariant case $p=\frac{n}{s}$. By \cite{BrezisNirenbergI}, it is known that the standard notions of homotopy can be extended to $W^{s,\frac{n}{s}}$-Sobolev maps, see also \cite[Section 2]{kasia2020minimal} for an overview.

For $s \in (0,1)$ set
\begin{equation}\label{def:infenergy}
 \#_s \alpha \coloneqq \inf_{u \in W^{s,\frac ns}(\S^n,\S^\ell), u \in \alpha} \mathcal{E}_{s,\frac ns}(u), \quad \alpha \in \pi_{n}(\S^\ell).
\end{equation}
In the case of maps between spheres of same dimension $u\colon \S^n\to \S^n$, i.e., when the homotopy classes are given by their degree we instead write
\begin{equation}\label{def:infenergysphere}
 \#_s d \coloneqq \inf_{u \in W^{s,\frac ns}(\S^n,\S^\ell), \deg(u)=d } \mathcal{E}_{s,\frac ns}(u), \quad d \in \Z.
\end{equation}

In general the question whether $\#_s \alpha$ is attained is rather involved even in the the local case $s=1$, see, e.g., \cite{Sucks1, DK98, R98} or, for $n=1$ and $s\in(0,1)$, see \cite[Chapter 12]{BM-book}. In \cite{kasia2020minimal} we showed\footnote{
The case $s\le \frac12$, $n=1$ was not treated in \cite{kasia2020minimal} but is is covered in \cite{K2025}.}
\begin{theorem}\label{th:sucks}
For any $\ell,n \geq 1$, with either $(\ell,n) = (1,1)$ or $\ell \geq 2$, $s \in (0,1)$. There exists a generating set $X_s \subset \pi_{n}(\S^\ell)$ such that for any $\alpha \in X_s$ the infimum $\#_s \alpha$ is attained.
\end{theorem}

In this work we are interested in the stability of such results as $s$ changes.

Our first main result is that one can choose the generating set $X_s$ from \Cref{th:sucks} \emph{locally stable} as $s$ varies. More precisely we have,

\begin{theorem}\label{th:generator2}
Fix $n,\ell \geq 1$ with either $(\ell,n) = (1,1)$ or $\ell \geq 2$. Let $\Lambda > 0$ and set for $s \in (0,1)$
\[
 X_{s} \coloneqq \left \{\alpha \in \pi_n(\S^\ell)\colon \quad \text{there exists a $W^{s,\frac{n}{s}}(\S^n,\S^\ell)$-minimizer $u$ in $\alpha$ and $[u]_{W^{s,\frac{n}{s}}(\S^n,\S^\ell)} \leq \Lambda$} \right \}.
\]
Then for any $t \in (0,1)$ there exists $\delta > 0$ such that
\[
Y \coloneqq \bigcap_{s \in (t-\delta,t+\delta)} X_{s}
\]
spans the same set as $X_s$, i.e., $X_{s} \subset {\rm span} Y$, for each $s \in (t-\delta,t+\delta)$.
\end{theorem}
See \Cref{th:generator2v2} for a more precise statement with respect to the dependencies of $\delta$.

Let us stress that \Cref{th:generator2} does \emph{not} imply that
\[
 X_{s} \coloneqq \{\alpha \in \pi_{n}(\S^\ell)\setminus \{0\}\colon \quad \#_s \alpha \text{ is attained}\}
\]
is unchanged as $s$ varies. Rather, it says that we can choose the set of attained generators of $\pi_{n}(\S^\ell)$ locally stable. In particular, for $n=\ell$ (when the homotopy class is identified with the degree), in principle, it could be possible that
\[
 X_s = \begin{cases}
          \{1\} \quad &s < 1/2,\\
          \Z\setminus \{0\} \quad &s = 1/2,\\
          \{1,2,-3\} \quad &s \in (1/2,3/4),\\
          \{2,-3\} \quad &s \in (3/4,1).
       \end{cases}
\]

An important ingredient in the aforementioned \Cref{th:generator2}, and our second main result, is the following continuity result for the map $s \mapsto \#_s \alpha$. By lower semicontinuity of the norm, it is elementary that for $\alpha \in \pi_{n}(\S^\ell)$ we have for any $t \in (0,1)$
\[
 \#_t \alpha \leq \liminf_{s \to t^+} \#_s \alpha.
\]
But actually, we have full continuity.
\begin{theorem}\label{th:continuousminenergy}
Assume $\ell,n \geq 1$, with either $(\ell,n) = (1,1)$ or $\ell \geq 2$. Let $\alpha \in \pi_{n}(\S^\ell)$. Then the map
\[
 s \mapsto \#_s \alpha, \quad \text{$s \in (0,1)$}
\]
is continuous.
\end{theorem}

Besides being a crucial ingredient for the proof of \Cref{th:generator2}, \Cref{th:continuousminenergy} has also several interesting corollaries that we discuss now.

Firstly, it is natural to expect that minimizers exist in the class of degree one maps in $W^{s,\frac{n}{s}}(\S^n,\S^n)$. Towards this, Berlyand--Mironescu--Rybalko--Sandier obtain in \cite[Lemma 3.1]{BMRS14} that for maps in $W^{\frac12,2}(\S^1,\S^1)$ minimizers are attained for any degree, see also \cite[Theorem 12.9]{BM-book}. Moreover, Mironescu proved a stability result \cite[Theorem 2]{M15}, which asserts that for $s \in [\frac{1}{2},\frac{1}{2}+\delta)$ degree one minimizers exist in $W^{s,\frac1s}(\S^1,\S^1)$. As a corollary of \Cref{th:continuousminenergy} we can extend the latter in the other direction.
\begin{corollary}\label{co:existenceclosetohalf}
For maps from $\S^1$ to $\S^1$, there exists $\delta>0$ such that
 \[
\#_s 1 =  \inf\{[u]_{W^{s,\frac {1}{s}}}^{\frac {1}{s}}\colon u\in W^{s,\frac{1}{s}}(\S^1,\S^1), \, \deg u=1\}
 \]
is attained for all $s\in(\frac 12-\delta,\frac12 + \delta)$.
\end{corollary}
This provides progress towards \cite[Open Problem 1]{Mironescu2007} and \cite[Open Problem 24.]{BM-book}.

More generally, if $\#_s \alpha$ is attained for an $\alpha \in \pi_{n}(\S^\ell)$ it is unclear
whether $\#_t \alpha$ is attained for $t \approx s$. \Cref{co:existenceclosetohalf} works for degree one maps, because they have the lowest energy level among nontrivial homotopy classes. This observation is true in any dimension and we have
\begin{corollary}\label{co:existenceclosetohalfweirdo}
Fix $n,\ell \geq 1$ with either $(\ell,n) = (1,1)$ or $\ell \geq 2$. Assume that for some $t \in (0,1)$ and $\alpha \in \pi_{n}(\S^\ell) \setminus \{0\}$ we have
\[
 \#_t \alpha \leq \#_t \beta \quad \forall \beta \in \pi_{n}(\S^\ell) \setminus \{0\}.
\]
Then not only $\#_t \alpha$ is attained by \cite{kasia2020minimal}, but also there exists $\delta>0$ such that $\#_s \alpha$ is attained for all $s \in (t-\delta,t+\delta)$.
\end{corollary}

For our next corollary of \Cref{th:continuousminenergy}, we consider the Bourgain--Brezis--Mironescu degree inequality, \cite[Theorem 0.6]{BBM2}, which says that for maps $u \in W^{s,\frac{n}{s}}(\S^n,\S^n)$
 \begin{equation}\label{ineq:BBM}
  \deg u \le C_{n,s} [u]_{W^{s,\frac ns}(\S^n,\S^n)}^{\frac ns}.
 \end{equation}
Let $\bar{C}_{n,s}$ be the minimal constant, i.e.,
\[
 \bar{C}_{n,s} \coloneqq \sup_{u \in W^{s,\frac{n}{s}}(\S^n,\S^n)} \frac{\deg u}{[u]_{W^{s,\frac ns}(\S^n,\S^n)}^{\frac ns}} < \infty.
\]
It is natural to discuss the continuity of the map $s \mapsto \bar{C}_{n,s}$.
\begin{corollary}\label{co:BBMconstant}
The map $s \mapsto \bar{C}_{n,s}$ is lower semicontinuous.

More precisely, for any $\Lambda > 0$ the map $s \mapsto \bar{C}_{n,s;\Lambda}$ defined by
\[
 \bar{C}_{n,s;\Lambda} \coloneqq \sup_{u \in W^{s,\frac{n}{s}}(\S^n,\S^n), 0<[u]_{W^{s,\frac{n}{s}} }\leq \Lambda} \frac{\deg u}{[u]_{W^{s,\frac ns}(\S^n,\S^n)}^{\frac ns}}
\]
is continuous.
%
%
\end{corollary}

Corresponding results hold for the Hopf degree, see \cite{SVS20},
\[
 (1-\frac{1}{4n},1) \ni s \mapsto \tilde {C}_{n,s} \coloneqq \sup_{u \in W^{s,\frac{n}{s}}(\S^{4n-1},\S^{2n}), 0<[u]_{W^{s,\frac{4n-1}{s}} }} \frac{[u]_{\pi_{4n-1}(\S^{2n})}}{[u]_{W^{s,\frac{ 4n-1}s}(\S^n,\S^n)}^{\frac{4n}{s}}},
\]
and more generally for maps representing rational homotopy groups of spheres, see \cite{PS23}.

We turn to the main ideas for \Cref{th:continuousminenergy}, and thus \Cref{th:generator2}. We use the following new ingredient
\begin{equation}\label{eq:aslkdjcvn}
\text{whenever $\#_s \alpha$ is attained, then $\#_t \alpha \leq \#_s \alpha +\eps$ for $t \approx s$},
\end{equation}
for the precise formulation see \Cref{co:comparisonenergy}. To obtain \eqref{eq:aslkdjcvn} we show that $W^{s,n/s}(\S^n,\S^\ell)$ minimizers actually belong \emph{globally} to $W^{s_1,\frac{n}{s_1}}(\S^n,\S^\ell)$ for an $s_1>s$, see \Cref{th:regularity}. H\"older regularity in this situation has been established in \cite{S15,MS18}, after pioneering work for $n=1$ and $s=\frac{1}{2}$ in \cite{DaLio-Riviere-1Dsphere}, but this is a \emph{local} result on domains where the $BMO$-norm is small. Smallness of the BMO-norm is of course not a scaling invariant property, indeed it depends heavily on the specific minimizer $u$ and one cannot deduce from it a uniform global property. Our higher regularity  result in a conformally invariant Sobolev space, \Cref{th:regularity}, is, on the other hand, uniform and independent of the specific minimizer $u$. Hence, using stability of the Sobolev norm $W^{s,n/s}$, \Cref{pr:energyclose}, we obtain \eqref{eq:aslkdjcvn}. Once we have \eqref{eq:aslkdjcvn}, the main results follow from combinatorial observations coupled with the Sacks--Uhlenbeck theory developed in \cite{kasia2020minimal} and the energy identity from \Cref{th:energyidentity}.

We conclude this introduction with a few remarks about possible generalizations of these results.
\begin{remark}
\leavevmode
\begin{enumerate}
\item The modulus of continuity in \Cref{th:continuousminenergy} and the $\delta$ in \Cref{th:generator2} are relatively easily to compute. They depend on the regularity theory gain (which can be calculated explicitly) and they get worse with large $\#_s \alpha$.
 \item In this paper, to ensure clarity in our presentation, we focus on the case when the target manifold is a sphere. Nonetheless, it should be easy to extend the results to the case of a compact Lie group in the target.
 \item It seems that an extension of our results to general target manifold is  more challenging --- the only obstacle is regularity, but it is unclear even for minimizing maps how to get scaling-invariant higher regularity of \Cref{th:regularity}.
 \item It would be interesting to study the limiting cases as $s \to 1^-$ and $s \to 0^+$.
 \end{enumerate}

\end{remark}

\subsection*{Notation}
We write $\alpha, \beta$, etc. to denote a homotopy group, $\S^n$ for the sphere in the domain, and $\S^\ell$ for the target sphere. For brevity, we write $\aleq$ whenever there is a constant $C$ (not depending on any
crucial quantity) such that $A\le C $B. Similarly, $A\approx B$ means $A \aleq B$ and $B\precsim A$.

\subsection*{Acknowledgement}
The work presented in this paper was conducted as part of the Thematic Research Programme \emph{Analysis and Geometry in Warsaw}, which received funding from the University of Warsaw via IDUB (Excellence Initiative Research University). Part of this work was carried out while K.M. and A.S. were visiting University of Bielefeld. We like to express our gratitude to the University for its hospitality. A.S. is an Alexander-von-Humboldt Fellow. A.S. is funded by NSF Career DMS-2044898. The project is co-financed by the Polish National Agency for Academic Exchange within Polish Returns Programme - BPN/PPO/2021/1/00019/U/00001 (K.M.). The project is co-financed by National Science Centre grant 2022/01/1/ST1/00021 (KM).

\section{Preliminary results}
Let us emphasize that some of the results of this section can be easily extended to general target manifolds. For brevity we restrict everything to sphere targets.

The first result is well-known and follows from the embedding of the critical Sobolev space  into BMO is the following, see, e.g., \cite[Lemma 2.10]{kasia2020minimal}.
\begin{proposition}\label{pr:minimalenergy}
For any $\ell,n \in \N$ there exists $\lambda =\lambda(\ell,n)> 0$ such that
whenever $s \in (0,1]$ and $\alpha \in \pi_{n}(\S^\ell) \setminus \{0\}$
\[
 \#_s \alpha \geq \lambda.
\]
\end{proposition}

In \cite[Theorem 1.2]{VS20} or \cite[Lemma 12.6.]{BM-book} the following is proven.
\begin{proposition}\label{pr:easyenergy}
Whenever $s \in (0,1]$ and $\alpha$ has a free homotopy group decomposition into $(\alpha_i)_{i=1}^N$ then
\[
 \#_s \alpha \leq \sum_{i=1}^{N}\#_s \alpha_i.
\]
\end{proposition}

In the case of  $W^{\frac12,2}(\S^1,\S^1)$ maps, minimizers exists for each degree and their exact energy is known. Precisely, by \cite[Lemma 3.1]{BMRS14}, see also \cite[Theorem 12.9 \& Theorem 12.10]{BM-book},
\begin{theorem}\label{th:BBMfors=half} For maps from $\S^1$ to $\S^1$ we have
\begin{equation}\label{eq:BBMfors=half}
 \#_{\frac12} d = 4\pi^2 |d|.
\end{equation}
Moreover, $\#_{\frac12} d$ is attained for all $d \in \Z$.
\end{theorem}

By results in \cite{VS20} and \cite{BBM2} we obtain that if we know that the energy of a map is bounded then the map can belong only to a finite subgroup of $\pi_n(\S^l)$.
\begin{theorem}\label{th:finitehomotopy}
Fix $\Lambda > 0$ and let $0<s_0<s_1 <1$, $n,\ell \geq 1$ with either $(\ell,n) = (1,1)$ or $\ell \geq 2$. Then there exist a \emph{finite} subgroup $\mathcal{Q} \subset \pi_n(\S^\ell)$ such that the following holds:

Whenever for some $s \in (s_0,s_1)$ the map $u \in W^{s,\frac{n}{s}}(\S^n,\S^\ell)$ satisfies
\[
 [u]_{W^{s,\frac{n}{s}}(\S^n,\S^\ell)} \leq \Lambda
\]
then $u \in \mathcal{Q}$.
\end{theorem}
\begin{proof}
In the case when $n=\ell$ the assertion follows from the degree estimate in \cite[Theorem~0.6]{BBM2}, since for any $s > 0$ we have
\[
 \abs{\deg u} \aleq [u]_{W^{s,\frac{n}{s}}(\S^n,\S^n)}^{\frac{n}{s}}.
\]

If $\pi_1(\S^\ell) = \{0\}$, i.e., $\ell \geq 2$, we have for any $\eps > 0$
\[
\begin{split}
 &\int_{\S^n}\int_{\S^n} \chi_{\{|f(x)-f(y)|>\eps\}} \frac{1}{|x-y|^{2n}}\dx \dy\\
 &\leq \eps^{-\frac{n}{s_0}} \int_{\S^n}\int_{\S^n} \frac{|f(x)-f(y)|^{\frac{n}{s_0}}}{|x-y|^{2n}}\dx \dy \\
  &\leq \eps^{-\frac{n}{s_0}} 2^{\frac{n}{s_0}-\frac{n}{s_1}}\, \int_{\S^n}\int_{\S^n} \frac{|f(x)-f(y)|^{\frac{n}{s}}}{|x-y|^{2n}}\dx \dy \\
  &\leq \eps^{-\frac{n}{s_0}} 2^{\frac{n}{s_0}-\frac{n}{s_1}}\, \Lambda^{n}.
\end{split}
 \]
Hence the assumption in \cite[Theorem 1.4]{VS20} is satisfied and we may conclude.
\end{proof}

\subsection{Continuity for the \texorpdfstring{$s \mapsto W^{s,\frac{n}{s}}$}{Sobolev}-norm}
We need the following continuity result for the fractional Sobolev norm.

\begin{proposition}\label{pr:energyclose}
Fix $n,\ell \in \N$. Let $\Lambda > 0$, $t_0 > 0$, $t \in (t_0,1)$, and let $s_1 > t$. For any $\eps > 0$ there exists $\delta  =\delta(\eps,\Lambda,|s_1-t|,n)>0$ such that the following holds.

Assume that $u \colon \S^n \to \S^\ell$ satisfies
\[
[u]_{W^{s_1,\frac{n}{{s_1}}}(\S^n,\S^\ell)} \leq \Lambda.
\]
then
\[
 \sup_{r_1,r_2 \in (t-\delta,t+\delta)}\abs{[u]_{W^{r_1,\frac{n}{r_1}}(\S^n,\S^\ell)}^{\frac{n}{r_1}}-[u]_{W^{r_2,\frac{n}{r_2}}(\S^n,\S^\ell)}^{\frac{n}{r_2}}} \leq \eps.
\]
\end{proposition}

The proof is based on the following elementary lemma.
\begin{lemma}\label{le:basic}
For any $\eps > 0$, $\Gamma > 0$, $0<p_0<p_1$, there exists $\tilde \delta=\tilde{\delta}(p_0,p_1,\Gamma,\eps) > 0$ such that if $|p-q| < \tilde{\delta}$, $p,q \in [p_0,p_1]$ then
\begin{equation}\label{eq:ineqasdasd}
 |p-q| < \tilde{\delta} \quad \Rightarrow \quad |a^p-a^q| < \eps \quad \forall a \in [0,\Gamma].
\end{equation}
\end{lemma}
\begin{proof}
We may assume that $\eps \in (0,1)$ and $\Gamma \geq 1$, $p_0 \leq p < q \leq p_1$.

Set $\sigma \coloneqq (\frac{1}{2}\, \eps)^{\frac{1}{p_0}} \in (0,1)$, then 
\[
 |a^p-a^q| < 2\, \sigma^{p_0} \leq \eps \quad \forall a \in [0,\sigma].
\]
Moreover with the inequality $|1-e^t| \leq |t| e^{|t|}$ we find for $\Lambda_\eps \coloneqq \max_{a \in [\sigma,\Gamma]} |\ln a|$,
\[
 |a^p-a^q| = a^p |1-a^{q-p}| \leq \Gamma^{p_1} |1-a^{q-p}| \leq |q-p|\, \Gamma^{p_1}\, \Lambda_\eps e^{\Lambda_\eps 2p_1} \quad \forall a \in [\sigma,\Gamma].
\]
So if we set 
\[
 \tilde{\delta} \coloneqq \frac{1}{\Gamma^{p_1} \Lambda_\eps e^{\Lambda_\eps 2p_1}}
\]
We have shown that 
\[
 |a^p-a^q| < \eps \quad \forall |p-q| < \tilde{\delta}, \quad a \in [0,\Gamma].
\]
\end{proof}

\begin{proof}[Proof of \Cref{pr:energyclose}]
Pick some $\bar{s} \in (t,s_1)$ and take $ \frac{t-t_0}{2} <\delta < \frac{\bar{s}-t}{2}$ to be specified later. The relation between the numbers is now
\[
0<t_0<t-\delta <t<t+\delta<\bar s<s_1<1.
\]
Fix $r_1,r_2\in (t-\delta,t+\delta)$ such that $r_2 > r_1$. We have
\begin{equation}\label{eq:tousebasiclemma}
\begin{split}
 &\abs{[u]_{W^{r_1,\frac{n}{r_1}}(\S^n,\S^\ell)}^{\frac{n}{r_1}}-[u]_{W^{r_2,\frac{n}{r_2}}(\S^n,\S^\ell)}^{\frac{n}{r_2}}}\\
 &\leq \int_{\S^n}\int_{\S^n} \frac{\abs{|u(x)-u(y)|^{\frac{n}{r_1}}- |u(x)-u(y)|^{\frac{n}{r_2}}}}{|x-y|^{2n}}\dx \dy\\
 &= \int_{\S^n}\int_{\S^n} \frac{|u(x)-u(y)|^{\frac{n}{s_1}} \abs{|u(x)-u(y)|^{\frac{n}{r_1}-\frac{n}{s_1}}- |u(x)-u(y)|^{\frac{n}{r_2}-\frac{n}{s_1}}} }{|x-y|^{2n}}\dx \dy.
 \end{split}
\end{equation}
Set now
\[
 a \coloneqq |u(x)-u(y)|, \quad p \coloneqq \frac{n}{r_1}-\frac{n}{s_1},\quad q \coloneqq \frac{n}{r_2}-\frac{n}{s_1}.
\]
Since $|u| \equiv 1$ we have $a \in [0,2]$. Also $p,q \in [p_0,p_1]$ for
\[
 p_0 = \frac{n}{t+\delta}-\frac{n}{s_1} \geq \frac{n}{\bar{s}} - \frac{n}{s_1}>0 \quad \text{and} \quad  p_1 = \frac{n}{t-\delta}-\frac{n}{s_1} \leq \frac{n}{t_0} - \frac{n}{s_1} <\infty.
\]
We observe
\[
 |p-q| =\frac{n}{r_1 r_2} \abs{r_2-r_1} \leq \frac{2n}{(t_0)^2} \delta.
\]
Hence choosing $\delta = \tilde{\delta}\frac{(t_0)^2}{2n}$, where $\tilde \delta $ is from \Cref{le:basic}, and combining \eqref{eq:tousebasiclemma} with \eqref{eq:ineqasdasd} we get
\[
\begin{split}
 \abs{[u]_{W^{r_1,\frac{n}{r_1}}(\S^n,\S^\ell)}^{\frac{n}{r_1}}-[u]_{W^{r_2,\frac{n}{r_2}}(\S^n,\S^\ell)}^{\frac{n}{r_2}}}
 &\leq \sup_{a \in [0,2]} \abs{a^p-a^q}\, \int_{\S^n}\int_{\S^n} \frac{|u(x)-u(y)|^{\frac{n}{s_1}}  }{|x-y|^{2n}}\dx \dy\\
 &\leq \eps\int_{\S^n}\int_{\S^n} \frac{|u(x)-u(y)|^{\frac{n}{s_1}}  }{|x-y|^{2n}}\dx \dy,
 \end{split}
\]
as desired.

\end{proof}

\section{Existence of minimizers and energy identity}
In this section we show how to deduce an energy identity using \cite{kasia2020minimal} and \cite{VS20}. We begin with recalling the following Lemma, which we will combine later with \Cref{th:finitehomotopy}.
\begin{lemma}[{\cite[Lemma 7.7]{kasia2020minimal}}]\label{le:sucks}
Fix $n,\ell \in \N$ with either $(\ell,n) = (1,1)$ or $\ell \geq 2$, and $s \in (0,1)$.
There is a number $\theta=\theta(s,n,\ell)$ such that the following holds\footnote{As mentioned before the case $s\le \frac12$, $n=1$ was not treated in \cite{kasia2020minimal} but is is covered in \cite{K2025}.}.
Let $\alpha \in \pi_{n}(\S^\ell) \setminus\{0\}$. Then either $\#_s \alpha$ is attained or for any $\delta > 0$ there exist $\alpha_1,\alpha_2 \in \pi_{n}(\S^\ell) \setminus \{0\}$ (possibly depending on $\delta$) such that $\alpha = \alpha_1 + \alpha_2$,
 \begin{equation}\label{eq:energydecomp}
  \#_s\alpha_1 + \#_s \alpha_2 \le \#_s \alpha + \delta,
  \end{equation}
  and
  \begin{equation}\label{eq:energydecrease}
  \quad \theta<\#_s \alpha_i < \#_s\alpha - \frac{\theta}{2}, \quad \text{ for } i=1,2.
 \end{equation}
\end{lemma}

From \Cref{le:sucks} we can conclude the following existence and energy identity.
\begin{theorem}[Energy identity]\label{th:energyidentity}
Fix $n,\ell \in \N$, with either $(\ell,n) = (1,1)$ or $\ell \geq 2$, and $s \in (0,1)$. For each $\alpha \in \pi_{n}(\S^\ell) \setminus \{0\}$ there exists a finite sequence $(\alpha_i)_{i=1}^N \subset \pi_{n}(\S^\ell)\setminus \{0\}$ such that
\begin{enumerate}
 \item $\alpha = \sum_{i=1}^N \alpha_i$,
 \item $\#_s \alpha = \sum_{i=1}^N \#_s \alpha_i$,
 \item $\#_s \alpha_i$ are attained for each $i \in \{1,\ldots,N\}$.
\end{enumerate}
\end{theorem}
\begin{proof}
Fix $\Lambda \coloneqq \#_s \alpha+1$.


Fix $\tilde \delta \in (0,1)$. If $\#_s \alpha$ is attained then we are done. If $\#_s \alpha$ is not attained, we apply \Cref{le:sucks} and decompose $\alpha = \alpha_1 + \alpha_2$ with
\[\#_s\alpha_1 + \#_s \alpha_2 \le \#_s \alpha + 2^{-1}\tilde \delta \quad \text{and}\quad \theta<\#_s \alpha_i < \#_s\alpha - \frac{\theta}{2} \text{ for }i=1,2.
\]
If both $i\in{1,2}$, $\#_s \alpha_i$ are not attained we finish, if at least one is not attained, say $\#_s \alpha_1$ we apply again \Cref{le:sucks} with $\delta = 2^{-2}\tilde \delta$. We decompose $\alpha_1=\alpha_{1,1} + \alpha_{1,2}$ and obtain
\begin{equation}\label{eq:energydecomposition1}
 \#_s\alpha_{1,1} + \#_s \alpha_{1,2} + \#_s\alpha_{2} \le \#_s \alpha_1 + \#_s\alpha_{2} +2^{-2}\tilde \delta+ 2^{-1}\tilde \delta \le \#_s\alpha +\sum_{i=1}^2 2^{-i}\tilde \delta
 \end{equation}
 with
 \begin{equation}\label{eq:energydecrease1}
 \theta<\#_s \alpha_{1,i} < \#_s\alpha_1 - \frac{\theta}{2} < \#_s\alpha - \theta,
 \end{equation}
where $\alpha = \alpha_{1,1} + \alpha_{1,2} + \alpha_2$. With an abuse of notation we relabel $\alpha = \alpha_{1} + \alpha_{2} + \alpha_3$.

We apply \Cref{le:sucks} iteratively, whenever the minimizer of one of the decomposed terms is not attained --- at $\ell$-th step we take $\delta = 2^{-\ell} \tilde\delta$ and obtain a decomposition into $\ell + 1$ terms with
\begin{equation}
 \sum_{i=1}^{\ell+1} \#_s \alpha_i \le \#_s \alpha + \sum_{i=1}^{\ell +1}2^{-i}\tilde\delta.
\end{equation}
By \eqref{eq:energydecrease} after $2^{\ell-1}$-iterations there is an $i\in \{1,\ldots,2^{\ell-1}+1\}$ for which
\begin{equation}\label{eq:enerdydecrease2}
\theta < \#_s \alpha_i < \#_s \alpha - \ell \frac{\theta}{2} = \Lambda -1 -\ell \frac{\theta}{2}.
\end{equation}
If however $\ell> \frac{2}{\theta}(\Lambda-1-\theta)$ we get a contradiction in \eqref{eq:enerdydecrease2}. Hence, we may iterate at most $2^{L-1}$ times for an $L=L(\theta,\Lambda)$ obtaining a decomposition into $N_{\tilde \delta}$ terms, where $N_{\tilde \delta}\le 2^{L-1}+1$ for all $\tilde \delta \in (0,1)$ and
%
\[
 \sum_{i=1}^{N_{\tilde \delta}} \alpha_i = \alpha
\]
such that $\#_s \alpha_i$ must be attained (otherwise we would have continued the iteration of \Cref{le:sucks}). Moreover,
\[
 \sum_{i=1}^{N_{\tilde \delta}}\#_s \alpha_i \leq \#_s \alpha + \tilde \delta.
\]
We want to let $\tilde \delta \to 0$. Let us stress that as of now the decomposition of $\alpha_i$ depends on $\tilde \delta$.

By \Cref{th:finitehomotopy},
\[
 \{\beta \in \pi_{n}(\S^\ell)\colon \quad \#_s \beta \leq \#_s \alpha + 1\}
\]
is a finite set. Hence, there is only a finite number of possibilities of $\alpha_i=\alpha_i(\tilde \delta)$. Thus, there exists a sequence $\tilde \delta_k \to 0$ such that $\alpha_i(\tilde \delta_k)=\alpha_i(\tilde \delta_j)$ for all $j,k\in\N$. Thus we have
\[
 \sum_{i=1}^{N}\#_s\alpha_{i} \leq \#_s\alpha + \tilde \delta_k,
\]
where the left-hand side does not depend on $\tilde \delta_k$ anymore. Letting $k \to \infty$ we obtain
\[
 \sum_{i=1}^{N}\#_s\alpha_{i} \leq \#_s\alpha.
\]
On the other hand, since $\sum_{i=1}^N \alpha_i = \alpha$ we have by \Cref{pr:easyenergy},
\[
 \#_s\alpha \leq \sum_{i=1}^{N}\#_s\alpha_{i}.
\]
Combining the two last statements we obtain
\[
 \#_s\alpha =\sum_{i=1}^{N}\#_s\alpha_{i}.
\]

\end{proof}
Let us remark in passing that another strategy to obtain \Cref{th:energyidentity} would be to extend the methods \cite{DK98}, which mostly rely on the conformal invariance of the norms involved, and thus should be applicable to our case.

\section{Globally improved regularity in the conformal scaling}
In this section we show that $W^{s,\frac ns}$-minimizers actually belong \emph{globally} to $W^{\Theta s,\frac{n}{\Theta s}}(\S^n,\S^\ell)$ for a $\Theta>1$. To do so we adapt the strategy of \cite{S15}. It seems to us that we use a somehow unique feature of the Gagliardo-seminorm and the fractional $p$-Laplacian (i.e., we do not see how the argument would work for the classical $p$-Laplace): we use intrinsically a feature of differential stability that was observed for the fractional $p$-Laplace but is not known for the classical $p$-Laplace, see \cite{KMS15,S16}.

\begin{theorem}\label{th:regularity}
Fix $n, \ell \in \N$.
For every $0<s_0 <s_1 < 1$ there exists $\Theta > 1$ and for any $\Lambda > 0$ there is a  constant $C(\Lambda) = C(n,\ell,s_0,s_1,\Lambda) > 0$ such that the following holds:

If $s_0<s<s_1$ and $u \in W^{s,\frac{n}{s}}(\S^n,\S^\ell)$ is a $W^{s,\frac{n}{s}}(\S^n,\S^\ell)$-minimizer in its own homotopy group with
\[
 [u]_{W^{s,\frac{n}{s}}(\S^n,\S^\ell)} \leq \Lambda
\]
then
\[
 [u]_{W^{\Theta s,\frac{n}{\Theta s}}(\S^n,\S^\ell)} \leq C(\Lambda).
\]
\end{theorem}
\Cref{th:regularity} is a consequence of \Cref{pr:improvedest} and \Cref{pr:iwaniec} below.

\begin{remark}
\leavevmode
\begin{enumerate}
\item The study of the regularity of fractional harmonic maps was initiated by Da Lio--Rivi\`{e}re in their celebrated papers \cite{DaLio-Riviere-1Dmfd,DaLio-Riviere-1Dsphere}.
 \item The argument below applies generally to critical points, not only minimizers. We only state the a priori estimate versions, since (local) higher regularity was discussed in \cite{kasia2020minimal}.
\item In \cite{S15} it was proven that minimizers as in \Cref{th:regularity} are H\"older continuous, see also \cite{MS18}. Observe that, however, there is no hope to prove
\[
 [u]_{C^{0,\alpha}(\S^n,\S^\ell)} \aleq [u]_{W^{s,\frac{n}{s}}(\S^n,\S^\ell)}.
\]
Indeed, this can be simply disproved by conformal rescaling, concentrating the map $u$ into one-point. The right-hand side is conformally invariant, and thus does not change. However, the continuity on the left-hand side becomes worse and worse.
\item Similarly, there is no hope to obtain
\[
 [u]_{W^{s_1,p_1}(\S^n,\S^\ell)} \aleq [u]_{W^{s,\frac{n}{s}}(\S^n,\S^\ell)}.
\]
whenever $s_1p_1 > n$.
\item While \Cref{th:regularity}, does not imply continuity, and therefore seems to be a weaker result than \cite{S15,MS18}, it has the crucial advantage of being a global result on all of $\S^n$.
\end{enumerate}
\end{remark}

By the conformal invariance of the energy we may replace the domain $\S^n$ by $\R^n$ and assume $u\in \dot{W}^{s,\frac{n}{s}}(\R^n,\S^\ell)$.

\subsection{Improved global estimates for harmonic maps}

We begin with the Euler--Lagrange equations for $W^{s,\frac ns}$-harmonic maps, see, e.g., \cite{S15}. For any $\varphi \in L^\infty \cap \dot{W}^{s,\frac{n}{s}}(\R^n,\R^{\ell+1})$ a critical point $u \in W^{s,\frac{n}{s}}(\R^n,\S^\ell)$ of the energy $\mathcal{E}_{s,\frac ns}$ satisfies the equation
\begin{equation}\label{eq:pdemin}
 \int_{\R^n} \int_{\R^n} \frac{|u(x)-u(y)|^{\frac ns-2} (u(x)-u(y)) \wedge (u(x) \varphi(x)-u(y)\varphi(y))}{|x-y|^{2n}} \dx \dy = 0.
\end{equation}
Here $\wedge\colon \R^\ell \times \R^\ell$ denotes the wedge product in $\R^{\ell+1}$,
\[
 (u \wedge v)_{ij} = u^i v^j - u^j v^i, \quad \text{$i,j \in \{1,\ldots,\ell+1\}$}.
\]
The crucial result is that the equation for fractional $W^{s,\frac ns}$-harmonic maps improves globally.
\begin{proposition}\label{pr:improvedest}
Let $0 < s_0 < s_1 < 1$. There exists a $\theta \in (0,1)$ such that the following holds:

For any $s \in (s_0,s_1)$ and any $u\in\dot{W}^{s,\frac{n}{s}}(\R^n,\S^\ell)$, which is minimizing $W^{s,\frac ns}$-harmonic map in its own homotopy group we have
\[
 [(-\lap)^s_\frac ns u]_{W^{-\theta s,\frac{n}{n-\theta s}}(\R^n)} \leq C(s_0,s_1,n) [u]_{W^{s,\frac ns}(\R^n)}^{\frac ns},
\]
that is for any $\psi \in \dot{W}^{\theta s,\frac{n}{\theta s}}(\R^n,\R^{\ell+1})$,
\begin{equation}\label{eq:asdlkjhslf2}
\begin{split}
 \int_{\R^n} \int_{\R^n} &\frac{|u(x)-u(y)|^{\frac ns-2} (u(x)-u(y))\, \brac{\psi(x)-\psi(y)}}{|x-y|^{2n}}\dx \dy\\
 &\leq C(s_0,s_1,n) [u]_{W^{s,\frac ns}(\R^n)}^{\frac ns} [\psi]_{W^{\theta s,\frac{n}{\theta s}}(\R^n)}.
\end{split}
 \end{equation}
\end{proposition}

In the proof below we will frequently work with the fractional Laplacian $\laps{s}$ which can be defined as
\[
 \laps{s} f = c_{s} \mathcal{F}^{-1} \brac{|\xi|^s \mathcal{F} f},
\]
where $\mathcal{F}$ is the Fourier transform. When $s \in (0,1)$, with a different constant, we also have the representation
\[
 \laps{s} f(x) = c_s \int_{\R^n} \frac{f(y)-f(x)}{|x-y|^{n+s}} \dy.
\]
The inverse of the fractional Laplacian is the Riesz potential $\lapms{s}$, defined as
\[
 \lapms{s} f = c_{s} \mathcal{F}^{-1} \brac{|\xi|^{-s} \mathcal{F} f},
\]
or, for $s \in (0,n)$,
\begin{equation}\label{eq:rieszpotdef}
 \lapms{s} f(x) = c_{s} \int_{\R^n} |x-y|^{s-n} f(y) \dy.
\end{equation}
For mapping properties of the Riesz potential we refer the reader to standard literature, e.g., \cite{Stein,GrafakosCF,GrafakosMF, RunstSickel}.

In order to prove \Cref{pr:improvedest} we consider the following potential introduced in \cite{S15}:
\begin{equation}\label{eq:potential}
 T_{t} u(z) \coloneqq \int_{\R^n} \int_{\R^n} \frac{|u(x)-u(y)|^{\frac ns-2} (u(x)-u(y))\, (|x-z|^{t-n} - |y-z|^{t-n})}{|x-y|^{2n}} \dx \dy.
\end{equation}
Observe that for $t < s$ we have by the representation of the Riesz potential $I^t$, \eqref{eq:rieszpotdef},
\[
 \int_{\R^n} T_t u(z) \varphi(z) \dz =c\int_{\R^n} \int_{\R^n} \frac{|u(x)-u(y)|^{\frac ns-2} (u(x)-u(y))\, (\lapms{t}\varphi(x) - \lapms{t}\varphi(y))}{|x-y|^{2n}} \dx \dy.
\]
From the definition of $T_t$ and H\"older's inequality we have,
\begin{equation}\label{eq:TtestimateRiesz}
 \abs{\int_{\R^n} T_t u(z) \varphi(z) \dz} \aleq [u]_{W^{s,\frac ns}(\R^n)}^{\frac ns-1}\, [\lapms{t} \varphi]_{W^{s,\frac ns}(\R^n)}.
\end{equation}
Thus, if $t<s$, $T_{t} u$ is a tempered distribution for $u \in W^{s,\frac ns}(\R^n)$.

Observe that even as a distribution we have
\begin{equation}\label{eq:Ttintoorthandtang}
 \|T_t u\|_{L^{\frac{n}{n-t}}(\R^n)} \aleq \|u \cdot T_t u\|_{L^{\frac{n}{n-t}}(\R^n)} + \|u \wedge T_t u\|_{L^{\frac{n}{n -t}}(\R^n)},
\end{equation}
whenever the right-hand side is finite. In the next two lemmas we will estimate separately both terms on the right-hand side of \eqref{eq:Ttintoorthandtang}: the orthogonal projection $u \cdot T_t u$ and the tangential projection $u \wedge T_t u$ (orthogonal and tangential are meant with respect to the tangent space $T_u \S^{\ell}$).

\begin{lemma}\label{la:reg:ucdot}
Let $0 < s_0 < s_1 < 1$. There exists a $\theta \in (0,1)$ and a constant $C=C(s_0,s_1,n,\ell)>0$ such that the following holds:

For any $s \in (s_0, s_1)$, there exists $t < \theta s$ such that if $u$ is a $W^{s,\frac{n}{s}}(\R^n,\S^\ell)$-minimizing harmonic map in its own homotopy group, then for $T_t$ as in \eqref{eq:potential}
\begin{equation}\label{eq:orthogonalpart}
 \|u \cdot T_t u\|_{L^{\frac{n}{n-t}}(\R^n)} \leq C\, [u]_{W^{s,\frac{n}{s}}(\R^n)}^{\frac{n}{s}}.
\end{equation}

\end{lemma}
\begin{proof}
We argue similarly as in the proof of \cite[Lemma 6.5]{S15}.
We note that since $|u|=1$ we have \[(u(x) - u(y))\cdot u(z)= -\frac12(u(x) - u(y))\cdot (u(x)+u(y)-2u(z))\] and hence
\[
 \abs{u \cdot T_t u(z)} \aleq \int_{\R^n} \int_{\R^n} \frac{|u(x)-u(y)|^{\frac ns-1} |u(x)+u(y)-2u(z)|\, \abs{|x-z|^{t-n} - |y-z|^{t-n}}}{|x-y|^{2n}} \dx \dy.
\]
We observe that for $r \in (0,1)$, we have by \cite[Proposition 6.6]{ArminARMA},
\begin{equation}\label{eq:Prop6.6insomething}
 |u(x)-u(y)|\aleq |x-y|^{r}\, \brac{\mathcal{M} \laps{r} u(x) + \mathcal{M} \laps{r} u(y)},
\end{equation}
where $\mathcal{M}$ denotes the Hardy-Littlewood maximal function.

We follow the proof in \cite[Proposition 6.3]{S15} but replace the use of \cite[Proposition 6.2]{S15} by \eqref{eq:Prop6.6insomething} and consider the three regimes (similarly as in \cite[Proposition 6.1]{S15}):
\[
\begin{split}
&\{|x-y| \aleq \min\{|x-z|,|y-z|\}\}, \quad \text{ in this case $|x-z| \aeq |y-z|$};\\
&\{|x-z| \aleq \min\{|y-z|,|x-y|\}\}, \quad \text{ in this case $|y-z| \aeq |x-y|$};\\
&\{|y-z| \aleq \min\{|x-z|,|x-y|\}\}, \quad \text{ in this case $|x-z|\aeq |x-y|$}.
 \end{split}
\]
we obtain for $\tilde{t} \in (0,t)$ with $r+\tilde{t} \in (0,1)$

\begin{equation}\label{eq:threeuestimate}
\begin{split}
& |u(x)+u(y)-2u(z)|\, \abs{|x-z|^{t-n} - |y-z|^{t-n}}\\
&\aleq
\brac{\mathcal{M} \laps{r} u(x)+\mathcal{M} \laps{r} u(y)+\mathcal{M} \laps{r} u(z)}|x-y|^{r+\tilde t}k_{t-\tilde t,t}(x,y,z),
 \end{split}
 \end{equation}
where $\kappa_{\alpha,\gamma}(x,y,z)$ is given by
\begin{equation}\label{eq:eqeqeqeqeqee}
\begin{split}
 \kappa_{\alpha,\gamma}(x,y,z) &= \min\{|x-z|^{\alpha-n},|y-z|^{\alpha-n}\}\\
 &\quad + \brac{\frac{|y-z|}{|x-y|}}^{\gamma-\alpha}|y-z|^{\alpha-n}\chi_{\{|x-z| \aleq \min\{|y-z|,|x-y|\}\}}\\
 &\quad + \brac{\frac{|x-z|}{|x-y|}}^{\gamma-\alpha}|x-z|^{\alpha-n}\chi_{\{|y-z| \aleq \min\{|x-z|,|x-y|\}\}}.
\end{split}
 \end{equation}
 Moreover, we have by \eqref{eq:Prop6.6insomething}
 \begin{equation}\label{eq:plappart}
  |u(x)-u(y)|^{\frac ns-1} \aleq |x-y|^{r(\frac ns -1)}\brac{\brac{\mathcal M \laps{r}u(x)}^{\frac ns -1} + \brac{\mathcal M \laps{r}u(y)}^{\frac ns -1}}.
 \end{equation}
Combining \eqref{eq:plappart} with \eqref{eq:threeuestimate} we obtain for a $\varphi \in C_c^\infty(\R^n)$, $\|\varphi\|_{L^{\frac{n}{t}}(\R^n)} \leq 1$
\begin{equation}\label{eq:aslkvxc}
\begin{split}
&\|u \cdot T_t u\|_{L^{\frac{n}{n-t}}(\R^n)} \\
&\aleq \int_{\R^n} u(z) \cdot T_t u(z) \varphi(z)\dz \\
&\aleq\int_{\R^n} \int_{\R^n} \int_{\R^n}\frac{|u(x)-u(y)|^{\frac ns-1} |u(x)+u(y)-2u(z)|\, \abs{|x-z|^{t-n} - |y-z|^{t-n}}}{|x-y|^{2n}} \abs{\varphi(z)} \dx \dy \dz\\
&\aleq \int_{\R^n} \int_{\R^n} \int_{\R^n} |x-y|^{r\frac ns+\tilde{t}-n-n} k_{t-\tilde t,t}(x,y,z)U_r(x,y,z) \abs{\varphi(z)}\dx \dy\dz,
\end{split}
 \end{equation}
where
\[
\begin{split}
 &U_r(x,y,z)\\
 &\coloneqq \brac{\brac{\mathcal{M} \laps{r} u(x)}^{\frac ns-1} + \brac{\mathcal{M} \laps{r} u(y)}^{\frac ns-1} } \brac{\mathcal{M} \laps{r} u(x)+\mathcal{M} \laps{r} u(y)+\mathcal{M} \laps{r} u(z)}.
\end{split}
 \]
Assuming $r\frac{n}{s}+\tilde{t}-n > 0$ we estimate further \eqref{eq:aslkvxc} with the help of \cite[Proposition 6.4]{S15} and H\"{o}lder's inequality
\begin{equation}\label{eq:estimatesofI}
\begin{split}
&\|u \cdot T_t u\|_{L^{\frac{n}{n-t}}(\R^n)} \\
&\aleq \max_{t_1+t_2+t_3 = r\frac ns -n+t} \int_{\R^n}\, \lapms{t_1} \brac{\mathcal{M} \laps{r} u(z)}^{\frac ns-1}\, \lapms{t_2} \brac{\mathcal{M} \laps{r} u(z)}\, \lapms{t_3} \abs{\varphi(z)} \dz\\
&\aleq \max_{t_1+t_2+t_3 = r\frac ns-n+t} \|\lapms{t_1} \brac{\mathcal{M} \laps{r} u}^{\frac ns-1}\|_{L^{\frac{n}{r\frac ns-r-t_1}}(\R^n)}\, \|\lapms{t_2} \brac{\mathcal{M} \laps{r} u}\|_{L^{\frac{n}{r-t_2}}(\R^n)} \|\lapms{t_3}|\varphi|\|_{L^{\frac{n}{t-t_3}}(\R^n)}.
\end{split}
\end{equation}
Assuming $r\frac{n}{s}-r-t_1 \geq -r-t+n > 0$, $t_2 \leq r\frac{n}{s}-n+t < r$, $t_3 \leq r\frac{n}{s}-n+t < t$
we apply twice Sobolev's inequality and the Maximal Theorem to get
\begin{equation}\label{eq:estimatesofIcontinued}
\begin{split}
\|u \cdot T_t u\|_{L^{\frac{n}{n-t}}(\R^n)}
&\aleq \|\brac{\mathcal{M} \laps{r} u}^{\frac ns-1}\|_{L^{\frac{n}{r\frac ns-r}}(\R^n)}\, \|{\mathcal{M} \laps{r} u}\|_{L^{\frac{n}{r}}(\R^n)} \|\varphi\|_{L^{\frac{n}{t}}(\R^n)}\\
&\aleq \|\laps{r} u\|_{L^{(\frac ns-1)\frac{n}{r\frac ns-r}}(\R^n)}^{\frac ns-1}\, \|{\laps{r} u}\|_{L^{\frac{n}{r}}(\R^n)} \|\varphi\|_{L^{\frac{n}{t}}(\R^n)}\\
&= \|\laps{r} u\|_{L^{\frac{n}{r}}(\R^n)}^{\frac ns} \|\varphi\|_{L^{\frac{n}{t}}(\R^n)}\\
&\aleq [u]_{W^{s,\frac ns}(\R^n)}^{\frac ns}.%
\end{split}
\end{equation}
In the last line we used $r<s$.

The conditions on $r$ and $t$ used in estimates \eqref{eq:estimatesofI} and \eqref{eq:estimatesofIcontinued} are
\[
r < s, \quad  n<\frac{n}{s}r+t <n+r, \quad r+t < n.
\]
If $n=1$ we choose $r=t$, then any $t \in (\frac{s}{1+s},\min\{s,\frac{1}{2}\})$ is admissible, so we can pick
\[
 t \coloneqq \frac{\frac{s}{1+s}+\min\{s,\frac{1}{2}\}}{2} \leq s \frac{\frac{1}{1+s}+1}{2} \leq s \frac{\frac{1}{1+s_0}+1}{2}\coloneqq s\theta.
 \]
%
Observe that we do not need to make a distinction between the case $\frac{n}{s} < 2$ and $\frac{n}{s} \geq 2$.

If $n \geq 2$ it is even easier, since $r+t<n$ becomes a trivial condition.
This finishes the proof.

%

\end{proof}

Now we estimate the second part on the right-hand side of \eqref{eq:Ttintoorthandtang}.
\begin{lemma}\label{la:reg:uwedge}
Let $0 < s_0 < s_1 < 1$. There exists a $\theta \in (0,1)$ such that the following holds:

For any $s \in (s_0, s_1)$, there exists $t < \theta s$ such that if $u$ is a $W^{s,\frac{n}{s}}(\mathbb{R}^n,\mathbb{S}^\ell)$-minimizing harmonic map in its own homotopy group, then for $T_t$ as in \eqref{eq:potential}
\begin{equation}\label{eq:tangentialpart}
 \|u \wedge T_t u\|_{L^{\frac{n}{n-t}}(\R^n)} \aleq [u]_{W^{s,\frac{n}{s}}(\R^n)}^{\frac{n}{s}}.
\end{equation}
\end{lemma}
\begin{proof}
We argue as in \cite[Proof of Lemma 3.5]{S15}.

By duality, there is some $\psi \in C_c^\infty(\R^n)$, $\|\psi\|_{L^{\frac{n}{t}}(\R^n)}\le 1$,
\[
 \|u \wedge T_t u\|_{L^{\frac{n}{n-t}}(\R^n)} \aleq \int_{\R^n} u \wedge T_t u\, \psi.
\]
Take $R \gg 1$ so that $\supp \psi \subset B(0,R)$. Let $\eta_R\in C_c^\infty(B(0,2R))$ be a cut-off function such that $\eta_R\equiv 1$ in $B(0,R)$ and set
\[
\begin{split}
 \varphi_{1,R} &\coloneqq \eta_R \lapms{t} \psi,\\
 \varphi_{2,R} &\coloneqq (1-\eta_R) \lapms{t} \psi.
\end{split}
 \]
Then
\[
 \int_{\R^n} u \wedge T_t u\, \psi = \int_{\R^n} u \wedge T_t u\, \Ds{t} \varphi_{1,R} + \int_{\R^n} u \wedge T_t u\, \Ds{t} \varphi_{2,R}.
\]
We observe that with a constant independent of $R \gg 1$
\[
 \|\Ds{t} \varphi_{1,R}\|_{L^{\frac{n}{t}}(\R^n)} \aleq \|\psi\|_{L^{\frac{n}{t}}(\R^n)} \leq 1,
\]
hence
\[
 \begin{split}
 \int_{\R^n} u \wedge T_t u\, \psi
 &\leq \sup_{\varphi \in C_c^\infty(\R^n),\, \|\Ds{t} \varphi\|_{L^{\frac{n}{t}}(\R^n)} \aleq 1}  \int_{\R^n} u \wedge T_t u\, \Ds{t} \varphi\\
 &\quad +\limsup_{R \to \infty}\abs{\int_{\R^n} u \wedge T_t u\, \Ds{t} \varphi_{2,R}}.
 \end{split}
\]
For the second term on the right-hand side we observe that similarly as in \eqref{eq:TtestimateRiesz}, for suitably small $\eps > 0$
\[
 \begin{split}
 \abs{\int_{\R^n} u \wedge T_t u\, \Ds{t} \varphi_{2,R}}
&\aleq [u]_{W^{s,\frac{n}{s}}(\R^n)}^{\frac ns-1}\, [\lapms{t} (u\wedge \Ds{t} \varphi_{2,R})]_{W^{s,\frac{n}{s}}(\R^n)}\\
&\aleq [u]_{W^{s,\frac{n}{s}}(\R^n)}^{\frac ns-1}\, \|\Ds{s-t+\eps} (u\wedge \Ds{t} \varphi_{2,R})\|_{L^{\frac{n}{s+\eps}}(\R^n)}\\
&\aleq[u]_{W^{s,\frac{n}{s}}(\R^n)}^{\frac ns-1}\, (1+[u]_{W^{s,\frac{n}{s}}(\R^n)})\,  \|\Ds{s+\eps} \varphi_{2,R}\|_{L^{\frac{n}{s+\eps}}(\R^n)},
 \end{split}
\]
where in the second estimate we used an embedding in Triebel--Lizorkin spaces, see
\cite[Theorem 2.2.3]{RunstSickel}.

Now we observe that, due to the support of $1-\eta_R$ and $\psi$ we have for some $\sigma > 0$
\[
\begin{split}
 \|\Ds{s+\eps} \varphi_{2,R}\|_{L^{\frac{n}{s+\eps}}(\R^n)}
 &\aleq R^{-s-\eps} \|\lapms{t} \psi\|_{L^{\frac{n}{s+\eps}}} + \|(1-\eta_R) \laps{s+\eps-t} \psi\|_{L^{\frac{n}{s+\eps}}}\\
 &\aleq R^{-\sigma} C(\supp \psi) \|\psi\|_{L^{\frac{n}{t}}} \xrightarrow{R \to \infty} 0.
\end{split}
 \]

Thus, for some $\varphi \in C_c^\infty(\R^n)$, $\|\Ds{t} \varphi\|_{L^{\frac{n}{t}}(\R^n)} \leq 1$,
\begin{equation}\label{eq:threetermestimate}
\begin{split}
 \|u \wedge T_t u\|_{L^{\frac{n}{n-t}}(\R^n)}
 &\aleq \int_{\R^n} u \wedge T_t u\, \Ds{t} \varphi\\
 &=-\int_{\R^n} \varphi\, \Ds{t} u \wedge T_t u + \underbrace{\int_{\R^n} \Ds{t}(\varphi\, u) \wedge T_t u}_{\overset{\eqref{eq:pdemin}}{=}0}-\int_{\R^n} H_{\Ds{t}} (\varphi,u) \wedge T_t u.
  \end{split}
\end{equation}
Here we use the Leibniz term notation
\[
 H_{\Ds{t}}(f,g) \coloneqq \Ds{t} (fg) - f\Ds{t} g - \Ds{t}f\ g.
\]
For the last term of \eqref{eq:threetermestimate} we observe
that similarly as in \eqref{eq:TtestimateRiesz} for a suitably small $\eps  >0$

\begin{equation}\label{eq:threetermestimatethird}
\begin{split}
 \int_{\R^n} H_{\Ds{t}} (\varphi,u) \wedge T_t u
 &\aleq[u]_{W^{s,\frac{n}{s}}(\R^n)}^{\frac ns-1}\, [\lapms{t} H_{\Ds{t}} (\varphi,u)]_{W^{s,\frac{n}{s}}(\R^n)}\\
&\aleq[u]_{W^{s,\frac{n}{s}}(\R^n)}^{\frac ns-1}\, \|\Ds{s-t+\eps} H_{\Ds{t}} (\varphi,u)\|_{L^{\frac{n}{s+\eps}}(\R^n)}\\
&\aleq[u]_{W^{s,\frac{n}{s}}(\R^n)}^\frac ns\, \|\Ds{t} \varphi\|_{L^{\frac{n}{t}}(\R^n)}.
\end{split}
 \end{equation}
The last line works as long $s-t +\eps < t$.

For the remaining estimates we abbreviate
\[
\text{diff}_{\frac ns}u(x,y)=|u(x)-u(y)|^{\frac ns-2}(u(x)-u(y)).
\]
By the representation of the Riesz potential, \eqref{eq:rieszpotdef}, and \eqref{eq:pdemin} we have
\begin{equation}\label{eq:threetermestimatefirst}
\begin{split}
 &\int_{\R^n} \varphi\, \Ds{t} u \wedge T_t u\\
 &=\int_{\R^n}\int_{\R^n} \frac{\text{diff}_{\frac ns}u(x,y) \wedge \brac{\lapms{t} \brac{\varphi\, \Ds{t} u}(x)-\lapms{t} \brac{\varphi\, \Ds{t} u}(y)}}{|x-y|^{2n}}\dx \dy\\
 &=\int_{\R^n}\int_{\R^n} \frac{\text{diff}_{\frac ns}u(x,y) \wedge \brac{\lapms{t} \brac{\varphi\, \Ds{t} u}(x)-\lapms{t} \brac{\varphi\, \Ds{t} u}(y) - \frac{1}{2}(u(x)-u(y))(\varphi(x)+\varphi(y))}}{|x-y|^{2n}}\dx \dy.
 \end{split}
\end{equation}
Exactly as in the first lines of the proof of \cite[Lemma 6.6]{S15} we have
\[
\begin{split}
 &\abs{\lapms{t} \brac{\varphi\, \Ds{t} u}(x)-\lapms{t} \brac{\varphi\, \Ds{t} u}(y) - \frac{1}{2}(u(x)-u(y))(\varphi(x)+\varphi(y))}\\
 &\aleq\int_{\R^n} \abs{|x-z|^{t-n}-|y-z|^{t-n}}\, \abs{\Ds{t} u(z)} \abs{\varphi(x)+\varphi(y)-2\varphi(z)} \dz.
\end{split}
 \]
Thus,
\[
\begin{split}
 &\abs{\int_{\R^n} \varphi\, \Ds{t} u \wedge T_t u}\\
 &\aleq\int_{\R^n}\int_{\R^n}\int_{\R^n} \frac{|u(x)-u(y)|^{\frac ns-1} \abs{|x-z|^{t-n}-|y-z|^{t-n}}\, \abs{\Ds{t} u(z)} \abs{\varphi(x)+\varphi(y)-2\varphi(z)}}{|x-y|^{2n}}\dx \dy \dz.
\end{split}
 \]
This is the same situation as in \eqref{eq:aslkvxc}: the role of $\varphi \in L^{\frac{n}{t}}$ in \eqref{eq:aslkvxc} is taken here by $\Ds{t} u \in L^{\frac{n}{t}}$, and the role of $|u(x)+u(y)-2u(z)|$ is taken by $|\varphi(x)+\varphi(y)-2\varphi(z)|$ and observe that $\Ds{t} \varphi \in L^{\frac{n}{t}}$. As was discussed there we can pick $r \leq t$, and thus we have
\[
\begin{split}
 \abs{\int_{\R^n} \varphi\, \Ds{t} u \wedge T_t u}
 &\aleq \|\laps{r} u\|_{L^{\frac{n}{r}}(\R^n)}^{\frac ns-1} \|\laps{r} \varphi \|_{L^{\frac{n}{r}}(\R^n)}^{\frac ns-1} \|\Ds{t} u\|_{L^{\frac{n}{t}}(\R^n)}\\
 & \aleq [u]_{W^{s,\frac{n}{s}}(\R^n)}^\frac ns\, \|\Ds{t} \varphi\|_{L^{\frac{n}{t}}(\R^n)}.
\end{split}
 \]
We can conclude. (Again it is worth noting that the proof above does not need to distinguish between the case $\frac{n}{s} \geq 2$ and $\frac{n}{s} \leq 2$).
\end{proof}

We are now ready to proceed with the proof of the main result of this section. 
\begin{proof}[Proof of \Cref{pr:improvedest}]
Combining \eqref{eq:Ttintoorthandtang} with \eqref{eq:tangentialpart} and \eqref{eq:orthogonalpart} we get for a $t < \theta s$, where $t$ and $\theta $ are as in \Cref{la:reg:ucdot} and \Cref{la:reg:uwedge}
\begin{equation}\label{eq:Ttestimate}
\|T_t u\|_{L^{\frac{n}{n-t}}(\R^n)} \aleq [u]_{W^{s,\frac ns}(\R^n)}^{\frac ns}.
\end{equation}

By duality \eqref{eq:Ttestimate} implies for any $\varphi \in C^\infty \cap L^{\frac{n}{t}}(\R^n,\R^{\ell+1})$,
\[
 \int_{\R^n} \int_{\R^n} \frac{|u(x)-u(y)|^{\frac ns-2} (u(x)-u(y))\, \brac{\lapms{t} \varphi(x)-\lapms{t} \varphi(y)}}{|x-y|^{2n}}\dx \dy \aleq [u]_{W^{s,\frac ns}(\R^n)}^{\frac ns} \|\varphi\|_{L^{\frac{n}{t}}(\R^n)}.
\]
Thus for any $\psi \in C_c^\infty(\R^n)$
\begin{equation}\label{eq:asdlkjhslf}
 \int_{\R^n} \int_{\R^n} \frac{|u(x)-u(y)|^{\frac ns-2} (u(x)-u(y))\, \brac{\psi(x)-\psi(y)}}{|x-y|^{2n}}\dx \dy \aleq [u]_{W^{s,\frac ns}(\R^n)}^{\frac ns} \|\Ds{t} \psi\|_{L^{\frac{n}{t}}(\R^n)}.
\end{equation}
Using Sobolev embedding this implies for $t<t_2<s$
\[ \int_{\R^n} \int_{\R^n} \frac{|u(x)-u(y)|^{\frac ns-2} (u(x)-u(y))\, \brac{\psi(x)-\psi(y)}}{|x-y|^{2n}}\dx \dy \aleq [u]_{W^{s,\frac ns}(\R^n)}^{\frac ns} [\psi]_{W^{t_2,\frac{n}{t_2}}(\R^n)}.
\]
The constant depends on $|t-t_2|$, so by taking $\tilde{\theta}$ slightly larger than $\theta$ and $t_2 = \tilde{\theta} s$ we have $\tilde{\theta} s-t > (\tilde{\theta}-\theta)s_0$, so the constant can be chosen uniform and, by density, \cite[2.6.2. Proposition 1.]{RunstSickel}, we obtain \eqref{eq:asdlkjhslf2} for $\tilde{\theta}$.
\end{proof}

\subsection{A fractional version of Iwaniec's stability result}
A fractional version of Iwaniec's stability result was proposed in \cite{S16}. However, the result of \cite{S16} does not apply in our situation since it only considers stability in the differential direction, without adjusting the integrability. We need the latter, since we need to stay in the scaling invariant case. Hence, we employ a different version of the Iwaniec's stability result \cite[Theorem 13.2.1]{IM01} to obtain the following regularity result.
\begin{proposition}\label{pr:iwaniec}
For any $0 < s_0 < s_1 < 1$ there exists an $\eps_0 =\eps_0(s_0,s_1,n)> 0$ such that the following holds:

For any $s\in(s_0,s_1)$ and any $\Lambda > 0$ there exists a constant $C(\Lambda)$ such that if $u \in L^\infty \cap W^{s,\frac{n}{s}}(\R^n,\R^N)$ satisfies for a $t \in (s-\eps_0,s]$ and for any $\psi \in \dot{W}^{t,\frac{n}{t}}(\R^n,\R^N)$
\begin{equation}\label{eq:iwaniecpde}
 \abs{\int_{\R^n} \int_{\R^n} \frac{|u(x)-u(y)|^{\frac ns-2} (u(x)-u(y))\, \brac{\psi(x)-\psi(y)}}{|x-y|^{2n}}\dx \dy} \leq \Lambda [\psi]_{W^{t,\frac{n}{t}}(\R^n)},
\end{equation}
then for $r\coloneqq s\frac{n-t}{n-s} \geq s$
\[
 [u]_{W^{r,\frac{n}{r}}(\R^n)} \leq C(\Lambda,\eps_0,s_0,s_1).
\]
\end{proposition}

We first observe that for $p\coloneqq\frac ns$ and  $\eps \coloneqq p-\frac{n}{r} =\frac{n}{s} -\frac{n}{r}> 0$ we have $\frac{t}{1-\eps}=r$ and
\begin{equation}\label{eq:normWrnr}
 [u]_{W^{r,\frac{n}{r}}}^{\frac{n}{r}} = \int_{\R^n} \int_{\R^n} \frac{|u(x)-u(y)|^{p-2} (u(x)-u(y))\, |u(x)-u(y)|^{-\eps} (u(x)-u(y))}{|x-y|^{2n}}\dx \dy.
\end{equation}
We perform a de facto Hodge decomposition:
\begin{equation}\label{eq:hodge}
 |u(x)-u(y)|^{-\eps} (u(x)-u(y)) = A(x)-A(y) + G(x,y),
\end{equation}
where, in the terminology of \cite{MS18}, we choose $A$ such that
\begin{equation}\label{eq:choiceofA}
 (-\Delta)^{t} A(x) \coloneqq \div_{t} \brac{\frac{|u(x)-u(y)|^{-\eps} (u(x)-u(y))}{|x-y|^{t}}},
\end{equation}
that is for any $\varphi \in C_c^\infty(\R^n)$
\[
 (-\Delta)^{t} A[\varphi] = \int_{\R^n} \int_{\R^n} \frac{|u(x)-u(y)|^{-\eps} (u(x)-u(y))\, (\varphi(x)-\varphi(y))}{|x-y|^{n+2t}}\dx \dy.
\]
From the linear theory of partial differential equation we have:
\begin{lemma}\label{la:linearpdetheory}
For any $0 < t_0 < t_1 < 1$ there exists an $\eps_0 > 0$ such that whenever $t \in (t_0,t_1)$ and $A \in \dot{W}^{t,\frac{n}{t}}(\R^n,\R^N)$ as in \eqref{eq:choiceofA} exists and satisfies the estimate
\[
 [A]_{W^{t,\frac{n}{t}}(\R^n)} \aleq [u]_{W^{\frac{t}{1-\eps},(1-\eps)\frac{n}{t}}(\R^n)}^{1-\eps} = [u]_{W^{r,\frac{n}{r}}(\R^n)}^{1-\eps}
\]
with a constant independent of $\eps$ as long as $\eps \in [0,\eps_0]$. $A$ is unique up to constants.

\end{lemma}
\begin{proof}
Proceeding exactly as in \cite[Lemma A.1]{DLMS} we have the a priori estimate
\[
\begin{split}
(-\Delta)^{t} A[\varphi] &= \int_{\R^n} \int_{\R^n} \frac{|u(x)-u(y)|^{-\eps} (u(x)-u(y))\, (\varphi(x)-\varphi(y))}{|x-y|^{n+2t}}\dx \dy\\
&\leq [u]_{W^{\frac{t}{1-\eps},(1-\eps)\frac{n}{t}}(\R^n)}^{1-\eps}\, [\varphi]_{W^{t,\frac{n}{n-t}}(\R^n)}.
\end{split}
 \]
Using the identification via Triebel--Lizorkin spaces, \cite[Section~2]{RunstSickel}, we have
\[
 \begin{split}
 [A]_{W^{t,\frac{n}{t}}(\R^n)}
 \aeq[A]_{\dot{F}^{t,\frac{n}{t}}_{\frac{n}{t}}}
 \aeq[(-\Delta)^{t} A]_{\dot{F}^{-t,\frac{n}{t}}_{\frac{n}{t}}}
 \aeq[(-\Delta)^{t} A]_{\brac{\dot{F}^{t,\frac{n}{n-t}}_{\frac{n}{n-t}}}^\ast}
 \aleq [u]_{W^{\frac{t}{1-\eps},(1-\eps)\frac{n}{t}}(\R^n)}^{1-\eps}.
 \end{split}
\]
The constants depend only on $t_0$ and $t_1$ since $t \in (t_0,t_1)$.
In particular, $A$ exists since $(-\Delta)^t A \in \brac{\dot{F}^{t,\frac{n}{n-t}}_{\frac{n}{n-t}}}^\ast$. $A$ is unique up to constants, since $[A]_{W^{t,\frac{n}{t}}(\R^n)}=0$ implies that $A$ is a constant.
\end{proof}

From \Cref{la:linearpdetheory} and \eqref{eq:hodge} we have in particular
\[
 \brac{\int_{\R^n} \int_{\R^n} \frac{|G(x,y)|^{\frac{n}{t}}}{|x-y|^{2n}}\dx \dy}^{\frac{t}{n}} \aleq [u]_{W^{r,\frac{n}{r}}(\R^n)}^{1-\eps}.
\]
The latter estimate can, however, be improved.
\begin{proposition}\label{pr:smallness}
For any $0 < t_0 < t_1 < 1$ there exists an $\eps_0 > 0$ and a constant $C= C(t_0,t_1,\eps_0)$ such that for any $\eps \in (0, \eps_0)$, and $t \in (t_0,t_1)$ and $G$ as in \eqref{eq:hodge},
\[
 \brac{\int_{\R^n} \int_{\R^n} \frac{|G(x,y)|^{\frac{n}{t}}}{|x-y|^{2n}}\dx \dy}^{\frac{t}{n}} \leq C\, \abs{\eps} [u]_{W^{r,\frac{n}{r}}(\R^n)}^{1-\eps}.
\]
\end{proposition}

\begin{proof}We follow the approach in \cite[Theorem 13.2.1]{IM01}.
Fix $\eps_0$ to be specified later. By density, cf. \cite[2.6.2. Proposition 1.]{RunstSickel}, we may assume that $u \in C_c^\infty(\R^n)$ (observe that we can change $u$ by a constant without changing the definitions of $G$ and $A$).

For $z \in \C$, $|z| \leq \eps_0$ we set
\begin{equation}\label{eq:defGz}
 G_z(x,y) = \brac{A_z(x)-A_z(y)} - |u(x)-u(y)|^{z} (u(x)-u(y)),
\end{equation}
where $A_z$ is defined as the solution to
\begin{equation}\label{eq:defAz}
 (-\Delta)^{t} A_z[\varphi] = \int_{\R^n} \int_{\R^n} \frac{|u(x)-u(y)|^{z} (u(x)-u(y))\, (\varphi(x)-\varphi(y))}{|x-y|^{n+2t}}\dx \dy, \quad \forall \varphi\in C_c^\infty(\R^n).
\end{equation}
This is well-defined since $u \in C_c^\infty(\R^n)$ and the right-hand side is a linear functional on a Triebel--Lizorkin space.
Assume that for a $\lambda \in (0,\frac12)$ we have a $q \in (1,\infty)$ such that
\begin{equation}\label{eq:qcondition}2t-\frac{n}{q}(1+\Re(z)) \in( \lambda,1-\lambda).\end{equation} Let us remark already that later we will apply this to $q = \frac{n}{t} (1+\Re(z))$ for $|z| \leq \eps_0$, so that $\lambda$ and $\eps_0$ can be chosen depending only on $t_0$ and $t_1$. 

We get by \eqref{eq:defGz}
\begin{equation}\label{eq:firstestimate}
\brac{\int_{\R^n} \int_{\R^n} \frac{|G_z(x,y)|^{\frac{q}{1+\Re(z)}}}{|x-y|^{2n}}\dx \dy}^{\frac{1+\Re(z)}{q}}
\aleq  [u]_{W^{\frac{n}{q},q}(\R^n)}^{1+\Re(z)} + [A_z]_{W^{\frac{n}{q}(1+\Re(z)),\frac{q}{1+\Re(z)}}(\R^n)}.
\end{equation}
Arguing with the identification of Triebel--Lizorkin spaces as in \Cref{la:linearpdetheory} we have
\begin{equation}
[A_z]_{W^{\frac{n}{q}(1+\Re(z)),\frac{q}{1+\Re(z)}}(\R^n)}
\aeq   [A_z]_{\dot{F}^{\frac{n}{q}(1+\Re(z))}_{\frac{q}{1+\Re(z)},\frac{q}{1+\Re(z)}}(\R^n)}
\aeq   [(-\Delta)^{t} A]_{\dot{F}^{\frac{n}{q}(1+\Re(z))-2t}_{\frac{q}{1+\Re(z)},\frac{q}{1+\Re(z)}}(\R^n)}.
\end{equation}
Moreover,
\begin{equation}
 \dot{F}^{\frac{n}{q}(1+\Re(z))-2t}_{\frac{q}{1+\Re(z)},\frac{q}{1+\Re(z)}}(\R^n)
 = \brac{\dot{F}^{2t-\frac{n}{q}(1+\Re(z))}_{(\frac{q}{1+\Re(z)})',(\frac{q}{1+\Re(z)})'}(\R^n)}^\ast = \brac{W^{2t-\frac{n}{q}(1+\Re(z)),(\frac{q}{1+\Re(z)})'}(\R^n)}^\ast\eqqcolon X^\ast.
\end{equation}
Hence, by the equivalence of the norms (the constant depends on $\lambda$)
\begin{equation}\label{eq:secondestimate}
\begin{split}
&[A_z]_{W^{\frac{n}{q}(1+\Re(z)),\frac{q}{1+\Re(z)}}(\R^n)}\\
&\aeq  \sup_{[\varphi]_{X} \leq 1}\int_{\R^n} \int_{\R^n} \frac{|u(x)-u(y)|^{z} (u(x)-u(y))\, (\varphi(x)-\varphi(y))}{|x-y|^{n+2t}}\dx \dy\\
&\leq  \sup_{[\varphi]_{X} \leq 1}\int_{\R^n} \int_{\R^n} \frac{|u(x)-u(y)|^{1+\Re(z)}}{|x-y|^{\frac{n}{q}(1+\Re(z))}} \, \frac{\abs{\varphi(x)-\varphi(y)}}{|x-y|^{2t-\frac{n}{q}(1+\Re(z))}} \frac{\dx \dy}{|x-y|^{n}}\\
&\aleq[u]_{W^{\frac{n}{q},q}(\R^n)}^{1+\Re(z)}.
\end{split}
\end{equation}
We stress that for $|z| \leq \eps_0$ all of the constants above are independent of $z$ but depend on $\lambda$. Hence, combining the estimates \eqref{eq:firstestimate} with \eqref{eq:secondestimate} we obtain
\begin{equation}\label{eq:estimateofGnoeps}
 \brac{\int_{\R^n} \int_{\R^n} \frac{|G_z(x,y)|^{\frac{q}{1+\Re(z)}}}{|x-y|^{2n}}\dx \dy}^{\frac{1+\Re(z)}{q}}
\aleq  [u]_{W^{\frac{n}{q},q}(\R^n)}^{1+\Re(z)}.
\end{equation}
Fix now $\psi\colon \R^n\times \R^n\to \R^N$ such that
\begin{equation}\label{eq:psitestiw}
 \int_{\R^n} \int_{\R^n} \frac{|\psi(x,y)|^{q}}{|x-y|^{2n}}\dx \dy \leq 1
\end{equation}
and set
\[
 F_{\psi,u}(z) \coloneqq \int_{\R^n} \int_{\R^n} \frac{\langle G_z(x,y), |\psi(x,y)|^{q-2-\bar{z}} \psi(x,y)\rangle_{\C}}{|x-y|^{2n}}\dx \dy.
\]
Then
\[
 |F_{\psi,u}(z)| \leq\brac{\int_{\R^n} \int_{\R^n} \frac{|G_z(x,y)|^{\frac{q}{1+\Re(z)}}}{|x-y|^{2n}}\dx \dy}^{\frac{1+\Re(z)}{q}} \aleq [u]_{W^{\frac{n}{q},q}(\R^n)}^{1+\Re(z)},
\]
where the constants are independent of $q$ as long as the assumption \eqref{eq:qcondition} is satisfied.

We observe that $F_{\psi,u}(0) = 0$. Indeed, by the definition \eqref{eq:defAz} we have$(-\Delta)^t A_0 = (-\Delta)^t u$, hence $A = u+c$ for a constant $c$. This, by definition \eqref{eq:defGz}, implies $G_0(x,y) \equiv 0$.

Moreover, just as in \cite[Theorem 13.2.1]{IM01}, $z \mapsto F_{\psi,u}(z)$ is holomorphic: since $u \in C_c^\infty(\R^n)$ the map $\partial_{\bar{z}} F_{\psi,u}$ is well-defined, and we can compute explicitly that the Cauchy--Riemann equations are satisfied.


From Schwarz lemma for holomorphic functions we have for all $|z|\le \eps_0$
\begin{equation}\label{eq:schwarzlemmaest}
 \abs{F_{\psi,u}(z) } \aleq |z| [u]_{W^{\frac{n}{p},p}(\R^n)}^{1+\Re(z)}
\end{equation}
with constant independent of $\psi$ as long as \eqref{eq:psitestiw} is satisfied, the constant depends only on $\lambda$ and thus on $t_0$, $t_1$, and $\eps_0$.

Now take $q = \frac{n}{t}(1-\eps)$, $z=-\eps$, $\eps =\frac{n}{s} -\frac{n}{r}$, $r= s\frac{n-t}{n-s}$, so that $\frac{t}{1-\eps} =r$. By \eqref{eq:schwarzlemmaest} we get
\[
\begin{split}
 \brac{\int_{\R^n} \int_{\R^n} \frac{|G_{-\eps}(x,y)|^{\frac{n}{t}}}{|x-y|^{2n}}\dx \dy}^{\frac{t}{n}}
 =F_{\psi,u}(-\eps)\le \sup_{\psi \text{ as in \eqref{eq:psitestiw}}}F_{\psi,u}(-\eps) \aleq |\eps| [u]_{W^{\frac{n}{q},q}(\R^n)}^{1+\Re(z)}.
 \end{split}
\]
\end{proof}
We are ready to proceed with the proof of the main result of this section.
\begin{proof}[Proof of \Cref{pr:iwaniec}]
We have by \eqref{eq:normWrnr} and \eqref{eq:hodge}
\begin{equation}\label{eq:finalfinal}
\begin{split}
 [u]_{W^{r,\frac{n}{r}}(\R^n)}^{\frac{n}{r}}
 &= \int_{\R^n} \int_{\R^n} \frac{|u(x)-u(y)|^{\frac ns-2} (u(x)-u(y))\, \brac{A(x)-A(y)}}{|x-y|^{2n}}\dx \dy\\
  &\quad +\int_{\R^n} \int_{\R^n} \frac{|u(x)-u(y)|^{\frac ns-2} (u(x)-u(y))\, \brac{G(x,y)}}{|x-y|^{2n}}\dx \dy.
  \end{split}
  \end{equation}
By \eqref{eq:iwaniecpde} and \Cref{la:linearpdetheory} we have
\begin{equation}\label{eq:finalfinal1}
 \int_{\R^n} \int_{\R^n} \frac{|u(x)-u(y)|^{\frac ns-2} (u(x)-u(y))\, \brac{A(x)-A(y)}}{|x-y|^{2n}}\dx \dy \le \Lambda [A]_{W^{t,\frac{n}{t}}(\R^n)} \aleq \Lambda\, [u]_{W^{r,\frac{n}{r}}(\R^n)}^{1-\eps}.
\end{equation}
As for the second term on the right-hand side of \eqref{eq:finalfinal} we note that $\frac rn (\frac ns -1) + \frac tn=1$. Hence, by H\"{o}lder's inequality, \Cref{pr:smallness}, and the observation $\frac ns - \eps = \frac nr$
  \begin{equation}\label{eq:finalfinal2}
   \begin{split}
 \int_{\R^n} \int_{\R^n} &\frac{|u(x)-u(y)|^{\frac ns-2} (u(x)-u(y))\, \brac{G(x,y)}}{|x-y|^{2n}}\dx \dy\\
 &\aleq [u]_{W^{r,\frac{n}{r}}(\R^n)}^{\frac ns-1}\,  \brac{\int_{\R^n} \int_{\R^n} \frac{|G(x,y)|^{\frac{n}{t}}}{|x-y|^{2n}}\dx \dy}^{\frac{t}{n}} \\
  &\aleq \abs{\eps} [u]_{W^{r,\frac{n}{r}}(\R^n)}^{\frac nr}.
  \end{split}
  \end{equation}
Combining \eqref{eq:finalfinal} with \eqref{eq:finalfinal1} and \eqref{eq:finalfinal2} we obtain
  \begin{equation}
   \begin{split}
 [u]_{W^{r,\frac{n}{r}}(\R^n)}^{\frac{n}{r}}
 &\aleq  \Lambda\, [u]_{W^{r,\frac{n}{r}}(\R^n)}^{1-\eps} + \abs{\eps} [u]_{W^{r,\frac{n}{r}}(\R^n)}^{\frac{n}{r}}.
 \end{split}
\end{equation}
Now $s-t \leq \eps_0$ implies $\eps = \frac{n}{s}-\frac{n}{r}= \frac ns\brac{\frac{s-t}{n-t}}\le C(s_0,s_1)\eps_0$, so for $\eps_0$ suitably small we can absorb and conclude
\[
 [u]_{W^{r,\frac{n}{r}}(\R^n)}^{\frac{n}{r}} \aleq C(\eps_0,\Lambda,s_1,s_0) .
\]
\end{proof}

\section{Continuous dependence}
The main observation is that the regularity theory of \Cref{th:regularity} combined with the stability \Cref{pr:energyclose} implies:
\begin{corollary}\label{co:comparisonenergy}
Fix $n,\ell \in \N$ with with either $(\ell,n) = (1,1)$ or $\ell \geq 2$. Fix $0<s_0 < s_1 < 1$ and $\Lambda > 0$, then for any $\eps > 0$ there exists $\delta > 0$ such that the following holds.

If $s\in(s_0,s_1)$ and $\#_s \alpha$ is attained with $\#_s \alpha \leq \Lambda$ for a homotopy class $\alpha \in \pi_{n}(\S^\ell)$.

Then, for any $\tilde{s} \in (s-\delta,s+\delta)$,
\[
 \#_{s} \alpha \geq \#_{\tilde{s}} \alpha - \eps.
\]
\end{corollary}
\begin{proof}
Let $u\in \alpha$ be the minimizer of $\mathcal E_{s,\frac ns}$. By \Cref{th:regularity} there is a $\Theta>0$ such that
\[
 [u]_{W^{\Theta s,\frac{n}{\Theta s}}} \leq C(s_0,s_1,\Lambda).
\]
Now we pick $\delta$ from \Cref{pr:energyclose} and conclude for any $\tilde{s} \in (s-\delta,s+\delta)$
\[
 \#_{s} \alpha  = [u]_{W^{s,\frac{n}{s}}(\S^n)}^{\frac{n}{s}} \geq [u]_{W^{\tilde{s},\frac{n}{\tilde{s}}}(\S^n)}^{\frac{n}{\tilde{s}}} - \eps \geq \#_{\tilde{s}} \alpha-\eps.
\]
\end{proof}

\subsection{Continuous dependence of minimal energy: Proof of Theorem~\ref{th:continuousminenergy}}
\Cref{th:continuousminenergy} is a pretty straightforward consequence of \Cref{co:comparisonenergy}.
\begin{proof}[Proof of \Cref{th:continuousminenergy}]

Fix $\eps>0$ and $\alpha \in \pi_{n}(\S^\ell)$. Take any smooth map $\bar{u} \in C^\infty(\S^n,\S^\ell)$ that represents $\alpha$ then we have, for any $0<s_0<s_1 <1$ the estimate
\[
 \sup_{t \in (s_0,s_1)} \#_t \alpha \leq C(\bar{u}).
\]
That is, we have a uniform energy bound that is needed later in the application of \Cref{co:comparisonenergy}.

In view of \Cref{th:energyidentity} there is an $N\in\N$ for which
\[
 \#_s \alpha = \sum_{i=1}^N \#_s \alpha_i \quad \text{for some $\sum_{i=1}^N \alpha_i = \alpha$ such that $\#_s \alpha_i$ is attained.}
\]
By \Cref{co:comparisonenergy} there is a $\delta>0$ such that for all $t\in(s-\delta,s+\delta)$ we have
\[
 \#_s \alpha =  \sum_{i=1}^N \#_s \alpha_i\geq \sum_{i=1}^N \brac{\#_t \alpha_i - \frac{\eps}{N}} \geq \#_t \alpha - \eps,
\]
where the last inequality is a consequence of \Cref{pr:easyenergy}
The converse inequality follows by reversing the role of $t$ and $s$.
\end{proof}

\subsection{Proof of Corollaries}

\begin{proof}[Proof of \Cref{co:existenceclosetohalf}]
Fix $0<s_0<1/2 < s_1 < 1$.
Assume that $\#_s 1$ is not attained for some $s\in(s_0,s_1)$. Then by \Cref{th:energyidentity},
\[
 \#_s 1 \geq \#_s d_{1,s} + \#_s d_{2,s}
\]
for degrees $d_{1,s},d_{2,s} \in \mathcal Q\setminus\{-1,0,1\}$ such that $\#_s d_{i,s}$ is attained for $i=1,2$. $\mathcal Q$ is a finite set of integers, in view of \Cref{th:finitehomotopy}.

By \Cref{th:continuousminenergy}, the family of maps
\[
 (s_0,s_1) \ni s \mapsto \#_s d, \quad d \in \mathcal Q
\]
are equicontinuous. So, for every $\eps>0$ there is a $\delta$ such that for every $s \in (\frac12-\delta,\frac12+\delta)$ we have
\[
 \#_{\frac12} 1 \geq \#_s 1 -\eps \geq  \#_s d_{1,s} + \#_s d_{2,s} -\eps \geq \#_{\frac12} d_{1,s} + \#_{\frac12} d_{2,s} - 3\eps.
\]
Combining this with \Cref{th:BBMfors=half} we obtain
\[
 4\pi^2 = \#_{\frac12} 1 \ge \#_{\frac12} d_{1,s} + \#_{\frac12} d_{2,s} - 3\eps \ge 8\pi^2 -3 \eps.
\]
For $\eps<\frac43 \pi^2$ this is a contradiction.
\end{proof}
In a very similar way as \Cref{co:existenceclosetohalf} we obtain
\begin{proof}[Proof of \Cref{co:existenceclosetohalfweirdo}]
The fact that $\#_{t}\alpha$ is attained is an immediate consequence of \cite[Lemma 7.7]{kasia2020minimal}. Arguing exactly as in the proof of \Cref{co:existenceclosetohalf} --- assuming that $\#_s \alpha$ is not attained, we obtain for $\beta_{1,s},\beta_{2,s} \in \pi_{n}(\S^\ell)\setminus \{0\}$
\[
 \#_t \alpha \ge \#_{s} \beta_{1,s} + \#_{s} \beta_{2,s} - \eps \ge \#_{t} \beta_{1,s} + \#_{t} \beta_{2,s} - 3\eps.
\]
However, by assumption $\#_t \alpha \le \#_{t} \beta_{i,s}$ for $i=1,2$. Hence we would get $\#_t \alpha \ge 2\, \#_t \alpha -3\eps$. This gives a contradiction with \Cref{pr:minimalenergy} for sufficiently small $\eps>0$.
\end{proof}

\begin{proof}[Proof of \Cref{co:BBMconstant}]
In view of \Cref{th:finitehomotopy}, for each $\Lambda > 0$ there exists $D \in \mathbb{N}$ such that
\[
 \begin{split}
 \bar{C}_{n,s;\Lambda} &\equiv \sup_{u \in W^{s,\frac{n}{s}}(\S^n,\S^n), 0<[u]_{W^{s,\frac{n}{s}} }\leq \Lambda} \frac{\deg u}{[u]_{W^{s,\frac ns}(\S^n,\S^n)}^{\frac ns}}\\
&= \max_{d \in \Z \setminus\{0\}, |d| \leq D} \sup_{u \in W^{s,\frac{n}{s}}(\S^n,\S^n), [u]_{W^{s,\frac{n}{s}}} \leq \Lambda, \deg u = d} \frac{d}{[u]_{W^{s,\frac ns}(\S^n,\S^n)}^{\frac ns}}\\
&= \max \left \{\frac{d}{\#_s d}\colon \quad |d| \leq D,\, d \neq 0\quad \text{$\exists u \in W^{s,\frac{n}{s}}(\S^n,\S^n)$ with $\deg u = d$ and $[u]_{W^{s,\frac{n}{s}}} \leq \Lambda$}\right \}.
 \end{split}
 \]
For each $d \in \Z \setminus \{0\}$ the map $s \mapsto \frac{d}{\#_s d}$ is continuous, by \Cref{th:continuousminenergy}.

Since
\[
 \bar{C}_{n,s} = \sup_{\Lambda > 0} \bar{C}_{n,s;\Lambda} = \lim_{\Lambda \to \infty} \bar{C}_{n,s;\Lambda}
\]
we have that $s \mapsto \bar{C}_{n,s}$ is lower semicontinuous.
%
%
\end{proof}


%

\subsection{Stability of generators of minimizing \texorpdfstring{$W^{s,\frac{n}{s}}$}{Wsns}-harmonic maps: Proof of Theorem~\ref{th:generator2}}

\Cref{th:generator2} is a consequence of the more precise
\begin{theorem}\label{th:generator2v2}
Fix $n,\ell \geq 1$ with either $(\ell,n) = (1,1)$ or $\ell \geq 2$. Let $\Lambda > 0$ and $0<t_0<t_1 < 1$. There exists a $\delta =\delta(\ell,n,t_0,t_1,\Lambda) \in (0,1)$ such that the following holds for $t \in (t_0,t_1)$.

Set
\[
 X_{s} \coloneqq \left \{\alpha \in \pi_n(\S^\ell)\colon \quad \text{there exists a $W^{s,\frac{n}{s}}(\S^n,\S^\ell)$-minimizer $u$ in $\alpha$ and $[u]_{W^{s,\frac{n}{s}}(\S^n,\S^\ell)} \leq \Lambda$} \right \}
\]
and
\[
Y \coloneqq \bigcap_{s \in (t-\delta,t+\delta)} X_{s}.
\]
Then $X_s$ is generated by $Y$, i.e., $X_{s} \subset {\rm span} Y$, for each $s \in (t-\delta,t+\delta)$.
\end{theorem}

\begin{proof}[Proof of \Cref{th:generator2v2}]
Let $\eps>0$ be fixed and chosen below, for this $\eps$ we take $\delta >0$ from \Cref{co:comparisonenergy}.

By \Cref{th:finitehomotopy} there exists a finite number $M \in \N$ such that if for some $s\in(t-\delta,t+\delta)$ $[u]_{W^{s,\frac ns}(\S^n)}<\Lambda$ then $u\in\mathcal Q \coloneqq\{\alpha_1,\ldots,\alpha_M\}\subset\pi_n(\S^\ell)$. In particular, $X_s\subset \mathcal Q$ for all $s\in(t-\delta,t+\delta)$.
%
Observe that for each $s \in (t-\delta,t+\delta)$ we have from \Cref{th:sucks}
\[
 X_s \text{ generates } \mathcal Q.
\]
Let us enumerate the elements of $Y =\bigcap_{s \in (t-\delta,t+\delta)} X_s$
\[
 Y = \{\gamma_1,\ldots,\gamma_{K}\} \quad \text{ for $K \leq M$}.
\]
Of course, for now $K = 0$ is a possibility.

We define $Z\subset \mathcal Q\setminus Y$ as the collection of homotopy groups $\beta$ where $\#_s \beta$ is not attained for at least one $s$, more precisely:
\[
 Z \coloneqq \left\{\alpha \in \pi_{n}(\S^\ell)\colon\ \#_{s} \alpha \leq \Lambda\quad \text{ for an } s\in(t-\delta,t+\delta)\right\} \setminus Y = \{\beta_1,\ldots,\beta_L\}.
\]
We note that $L+K\le M$. 
Now assume that for some $s \in (t-\delta,t+\delta)$ there is $\alpha$ such that $\alpha \in X_s \setminus Y$. Then there must be some $\tilde{s} \in (t-\delta,t+\delta)$ with $\alpha \not \in X_{\tilde{s}}$.
By \Cref{th:energyidentity} we find $\beta_{k} \in Z$ and $\gamma_{k} \in Y\setminus \{0\}$ such that $\#_{\tilde s} \beta_{k}$ and $\#_{\tilde s}\gamma_{k}$ are attained, $2\le A + B\le M$, and
\begin{equation}\label{eq:energydecomposition123}
\alpha = \sum_{k=1}^{A} \beta_{k} + \sum_{k=1}^{B}\gamma_{k}, \quad \#_{\tilde{s}} \alpha= \sum_{k=1}^{A} \#_{\tilde{s}} \beta_{k} + \sum_{k=1}^{B} \#_{\tilde{s}} \gamma_{k}.
\end{equation}
Using \eqref{eq:energydecomposition123} and applying twice \Cref{co:comparisonenergy} we obtain
\begin{equation}\label{eq:blahblahblahblahv2}
 \#_s \alpha \geq \#_{\tilde{s}} \alpha -\eps = \sum_{k=1}^{A} \#_{\tilde{s}} \beta_{k} + \sum_{k=1}^{B} \#_{\tilde{s}} \gamma_{k} -\eps \geq \sup_{r \in (t-\delta,t+\delta)}\brac{\sum_{k=1}^{A} \#_{r} \beta_{k} + \sum_{k=1}^{B} \#_{r} \gamma_{k}} - (M+1)\eps.
\end{equation}
In particular we have
\begin{equation}\label{eq:blahblahblahblah}
 \#_s \alpha \geq \#_{\tilde{s}} \alpha -\eps \geq \sum_{k=1}^{A} \#_{s} \beta_{k} + \sum_{k=1}^{B} \#_{s} \gamma_{k} - (M+1)\eps.
\end{equation}
Observe that this is a contradiction if any of $\beta_{k}$'s on the right-hand side is equal to $\alpha$ (assuming $\eps$ is small enough).

If $A=0$ then we are done because then $\alpha = \sum \gamma_k$ and each $\gamma_k \in Y$.

If $A>0$, then for each $\beta_{k}$ in \eqref{eq:blahblahblahblahv2} we have two possibilities: either $\#_s \beta_{k}$ is attained (but still $\beta_{k}\not\in Y$) or $\#_s \beta_{k}$ is not attained. Rearranging the terms we can assume that
\begin{itemize}
 \item for the terms $\beta_{1},\ldots,\beta_{{A_1}}$, $\#_s \beta_k$ is attained, i.e., $\beta_k \in X_s \setminus Y$,
\item for the terms $\beta_{{A_1+1}},\ldots,\beta_{{A}}$ the infimum $\#_s \beta_k$ is not attained, i.e., $\beta_k \not \in X_s$.
\end{itemize}
It is possible that $A_1=0$ or $A_1=A$.

In case $\beta_{k}\in X_s\setminus Y$ we find an $s_k\in (t-\delta,t+\delta)$ such that $\beta_{k}\notin X_{s_k}$ and we apply \eqref{eq:blahblahblahblah}.

In case $\beta_{k}\notin X_s$ we apply \Cref{th:energyidentity}.

We obtain first using \Cref{co:comparisonenergy} and then using \eqref{eq:blahblahblahblah} and  \Cref{th:energyidentity}
\begin{equation}\label{eq:anothersplit}
 \begin{split}
  \sum_{k=1}^A \#_s \beta_{k} 
  &\ge \sum_{k=1}^{A_1} \#_{ s_k} \beta_{k} + \sum_{k=A_1+1}^A \#_s \beta_{k} -M\eps\\
  &\geq \sum_{k=1}^{A_1} \brac{\sum_{j=1}^{A_{1,k}} \#_{s} \beta_{k_j} + \sum_{j=1}^{B_{1,k}}\#_{s} \gamma_{k_j}}
  + \sum_{k=A_1+1}^A \brac{\sum_{j=1}^{A_{\tilde 1,k}} \#_s \beta_{\tilde k_j} + \sum_{j=1}^{B_{\tilde 1,k}} \#_s \gamma_{\tilde k_j}} - (2M^2+M)\eps\\
  &= \sum_{i=1}^{\tilde A_1} \#_s \beta_{\sigma(i)} + \sum_{i=1}^{\tilde B_2} \#_s \gamma_{\zeta(i)} -(2M^2+M)\eps,
 \end{split}
  \end{equation}
  where $\tilde A_1,\tilde B_1\le 2M^2$, $\sigma(i)\in\{1,\ldots,L\}$, $\zeta(i)\in\{1,\ldots,K\}$ for each $i$, and at each step $2\le A_{1,k} + B_{1,k}\le M$. Now again, we can repeat the procedure. However, since there is at most $M$ elements in $Z$ we obtain that after a finite number of iterations of this procedure the same element $\beta_i$ would appear on the right-hand side and on the left-hand side of \eqref{eq:anothersplit}. This would give a contradiction for a sufficiently small $\eps$ (depending on $M$ and $\lambda$ from \Cref{pr:minimalenergy}). Hence, we must obtain at some moment a decomposition of $\alpha$ into terms belonging only to $Y$.
\end{proof}

\bibliographystyle{abbrv}%
\bibliography{bib}%
\end{document}